\documentclass[12pt,a4paper,oneside]{amsart}

\usepackage{amsfonts, amsmath, amssymb, amsthm, amscd, mathtools}
\usepackage{xcolor}

\usepackage{hyperref}
\hypersetup{
	colorlinks   = true, 
	urlcolor     = blue, 
	linkcolor    = blue, 
	citecolor    = red 
}
\newcommand{\Href}[2]{\hyperref[#2]{#1~\ref{#2}}}
\usepackage{anysize}
\marginsize{2cm}{2cm}{1.5cm}{1.5cm}

\usepackage{graphicx}
\usepackage{epstopdf}
\usepackage{caption}
\graphicspath{{Fig/}}
\usepackage{float}
\usepackage{url,hyperref}

\usepackage[matrix,arrow]{xy}
\usepackage{enumitem}
\usepackage{array}

\usepackage{multirow}

\newtheorem{thm}{Theorem}

\newtheorem{lem}{Lemma}[section]
\newtheorem{cor}{Corollary}[section]

\theoremstyle{definition}
\newtheorem{dfn}{Definition}[section]
\newtheorem{rem}[dfn]{Remark}

\def\R{{\mathbb R}}

\def\CP{{\mathbb C}{\mathrm P}}

\DeclareMathOperator{\arcsh}{arcsinh}

\DeclareMathOperator{\sn}{sn}

\def\phi{\varphi}
\def\epsilon{\varepsilon}

\def\mod{\mathrm{mod}\ }

\def\xin{f}
\def\etan{g}
\def\s{s}

\DeclareMathOperator{\arcch}{arccosh}
\DeclareMathOperator{\sh}{sinh}

\DeclareMathOperator{\sgn}{sgn}

\title{Orthodiagonal anti-involutive Kokotsakis polyhedra}

\author{Ivan Erofeev}
\address{%
	Centre for Synthetic and Systems Biology,\\
	School of Biological Sciences,\\
	University of Edinburgh, Edinburgh EH9 3BF, UK
}
\email{ivan.erofeev@ed.ac.uk}
\author{Grigory Ivanov }
\address{Department of Mathematics\\
	University of Fribourg\\
	Chemin du Mus\'ee 23\\
	CH-1700 Fribourg P\'erolles\\
	SWITZERLAND}
\address{Department of Higher Mathematics\\
 Moscow Institute of Physics and Technology\\
   Institutskii pereulok 9, Dolgoprudny, Moscow
region\\ 141700, Russia.}
\email{grigory.ivanov@unifr.ch}
\thanks{%
	G.I. was supported by the Swiss National Science Foundation
	grant 200021\_179133.}

\begin{document}
\maketitle

\begin{abstract}
We study the properties of  Kokotsakis polyhedra of  orthodiagonal anti-involutive type.
Stachel conjectured that  a certain resultant connected to a polynomial system  describing flexion of a Kokotsakis polyhedron must be reducible. Izmestiev \cite{izmestiev2016classification} showed that  a polyhedron of the orthodiagonal anti-involutive type  is the only possible candidate to disprove Stachel's conjecture. 
We show that the corresponding resultant is reducible, thereby confirming the conjecture. We do it in two ways:  by  factorization
of the corresponding resultant and providing  a simple geometric proof.
We describe   the space of parameters
for which  such a polyhedron exists  and show that this space is non-empty.
We  show that a Kokotsakis polyhedron of  orthodiagonal anti-involutive type is flexible and  give  explicit parameterizations in elementary functions and in elliptic functions of its flexion.
\end{abstract}

\section{Introduction}

A \emph{Kokotsakis polyhedron} is a polyhedral surface in $\R^3$ which consists
of one $n$-gon (the \emph{base}), $n$ quadrilaterals attached to every side of the $n$-gon,
and $n$ triangles placed between each pair of  two consecutive quadrilaterals.
Such surfaces are rigid in general.
However, there exist special classes of flexible surfaces,
that is, ones  that allow isometric deformations.
Flexible Kokotsakis polyhedra were studied by several authors: in the PhD thesis of Antonios Kokotsakis and in \cite{kokotsakis1933bewegliche}; 
Nawratil \cite{nawratil2010flexible} studied flexible Kokotsakis polyhedra
with triangular base; and in \cite{sauer1931flachenverbiegung} some
classes of  flexible Kokotsakis polyhedra with quadrilateral base were found.
The question of complete classification for the case $n = 4$ was open.
Using Bricard's approach \cite{bricard1897memoire}, Stachel and Nawratil
\cite{nawratil2011reducible, nawratil2012reducible, stachel2010kinematic}
with the use of  the following trigonometric substitution

\begin{figure}[h]
	\includegraphics{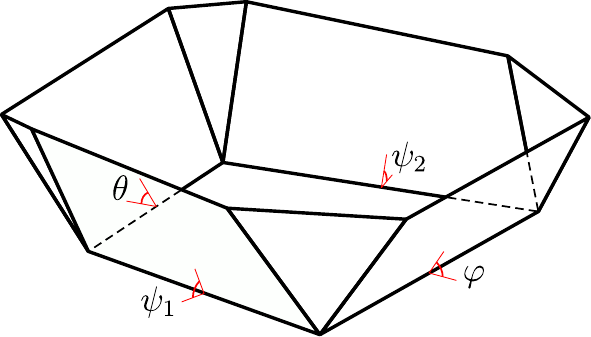}
	\caption{
	    \label{fig:intro-koto}
	    Notation of dihedral angles
	}
\end{figure}
\begin{equation}
	\label{eqn:Variables}
	\begin{gathered}
	z = \tan\frac{\phi}2,        \qquad w_1 = \tan\frac{\psi_1}2, \qquad
	u = \tan\frac{\theta}2,      \qquad w_2 = \tan\frac{\psi_2}2,
	\end{gathered}
\end{equation}

\noindent
expressed the dependencies between the dihedral angles 
$\phi$, $\psi_1$, $\theta$,  $\psi_2$
(see \Href{Figure}{fig:intro-koto}) as a polynomial system in $z$, $w_1$, $u$, $w_2$:
	\begin{equation}
	\label{eqn:PolSystem}
		\begin{aligned}
		P_1(z, w_1) &= 0, \quad & P_2(z, w_2) &= 0, \\
		P_3(u, w_1) &= 0, \quad & P_4(u, w_2) &= 0.
		\end{aligned}
	\end{equation} 
They studied this system with respect to different factorizations of the resultant 
$R_{12}(w_1, w_2)$ of $P_1$ and $P_2$ as polynomials in $z$,
and they classified all classes of flexible Kokotsakis polyhedra
with quadrilateral base for reducible $R_{12}$.
Stachel conjectured that there is no flexible Kokotsakis polyhedron
whose system has the irreducible resultant.

Considering a complexified version of system \eqref{eqn:PolSystem},
Izmestiev \cite{izmestiev2016classification} classified
all possible classes of flexible polyhedra disregarding reality and embeddability.
In particular, from the results of \cite{izmestiev2016classification} it follows that there is only one possible candidate 
for a flexible Kokotsakis polyhedron with the irreducible resultant $R_{12}$.
He called elements of this class Kokotsakis polyhedra 
of {\it orthodiagonal anti-involutive type} (OAI).

In this paper we study the properties of OAI Kokotsakis polyhedra.
We find an explicit parameterization in elementary functions of all  solutions
of the system \eqref{eqn:PolSystem} for this type
and show that there are solutions which may be embedded in  the real three-dimensional space.
We prove that the resultant $R_{12}$ is always reducible,
which confirms Stachel's conjecture. In addition, we provide 
a simple geometric proof of this fact in \Href{Theorem}{thm:Stachel_conj}
that directly follows from the results of \cite{izmestiev2016classification}.

The  rest of the article is organized as follows.
We give the definition of an OAI Kokotsakis polyhedron
in  \Href{Section}{sec:Def}
together with algebraic formulas for coefficients of the polynomial systems
required for further analysis.
However, we postpone an explanation of their geometric meaning to \Href{Section}{sec:geom_pr}.
\Href{Section}{sec:conf_space} is devoted to the description of
all solutions of the system \eqref{eqn:PolSystem} for an OAI Kokotsakis polyhedron,
which is the system \eqref{eqn:main_system_with_la_mu_signed}.
In \Href{Theorem}{thm:real_conf_space} of this Section,
we give an explicit parameterization
in elementary functions of all branches of the solution.
In other words, here we come up  with a description of all possible flexions of an OAI polyhedron
when all planar angles are given.
In \Href{Section}{sec:planar} we study a four parameter system
that describes planar angles of all flexible OAI Kokotsakis polyhedra.
We give a full parameterization of the solution set  of this system in terms of the angles
of the base quadrilateral and one additional parameter $\tau$
in \Href{Theorem}{thm:la-mu-parameterized}.
Summarizing the results, we introduce in \Href{Section}{sec:alg} an algorithm
for constructing a flexible OAI Kokotsakis polyderon for a given base quadrilateral without right angles.
In addition, we present results from the numerical screening of the space of parameters.
In \Href{Section}{sec:geom_pr}, we discuss the geometry behind
the definitions and   equations that we use,
show that there is a nice flattening effect in the case of the orthodiagonal anti-involutive surface,
and give a simple proof of Stachel's conjecture using it.
In \Href{Appendix}{sec:app_orth} we prove some technical results.

\begin{rem}
	We formulate all the theorems in such a manner that the reader can understand them
	right after reading the definitions of \Href{Section}{sec:Def}.
	This allows to use formulas from these theorems
	and construct a flexible polyhedron without going into the details of the proofs.
\end{rem}

\section{Definitions}
\label{sec:Def}
In this section we give an algebraic description
of a Kokotsakis polyhedron of the orthodiagonal anti-involutive type.
We do not discuss the nature of the written equations here.
A comprehensive explanations can be found in \cite{izmestiev2016classification}.
However, we give a brief explanation in \Href{Section}{sec:geom_pr},
where we show that Stachel's conjecture
is a simple consequence of the results from~\cite{izmestiev2016classification}.

\begin{figure}[h]
	\begin{center}
	\hspace{2cm}
	\begin{picture}(400,150)
	\put(20,0){\includegraphics[scale=1]{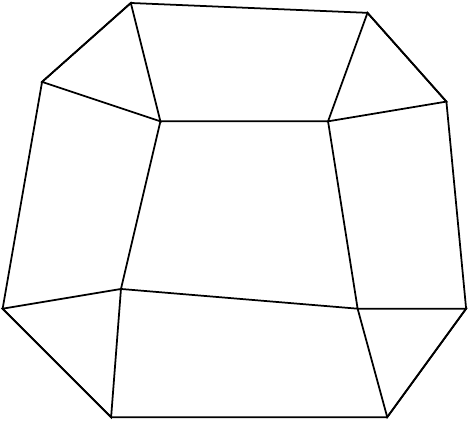}}
	\put(200,0){\includegraphics[scale=1]{Fig/2-def-plan-angles-eps-converted-to.pdf}}

	\put(56,40){$A_2$}
	\put(109,36){$A_1$}
	\put(65,77){$A_3$}
	\put(102,76){$A_4$}

	\put(157,32){$B_1$}
	\put(6,32){$B_2$}

	\put(290,37){$\delta_1$}
	\put(305,37){$\gamma_1$}
	\put(292,24){$\alpha_1$}
	\put(309,23){$\beta_1$}

	\put(222,42){$\gamma_2$}
	\put(240,42){$\delta_2$}
	\put(237,30){$\alpha_2$}
	\put(223,26){$\beta_2$}

	\put(232,79){$\gamma_3$}
	\put(247,77){$\delta_3$}
	\put(247,90){$\alpha_3$}
	\put(232,94){$\beta_3$}

	\put(298,81){$\gamma_4$}
	\put(285,77){$\delta_4$}
	\put(283,90){$\alpha_4$}
	\put(299,92){$\beta_4$}
	\end{picture}
	\caption{Planar angles in a Kokotsakis polyhedron.}

	\label{fig:NotPlanAngles}
	\end{center}
\end{figure}

Vertices of a Kokotsakis polyhedron and the values of its planar angles
at the interior vertices are denoted as in \Href{Figure}{fig:NotPlanAngles}.
Clearly, it is only a neighborhood of the {\it base} face $A_1A_2A_3A_4$ that matters:
replacing, say, the vertex $B_1$ by any other point on the half-line $A_1B_1$
does not affect the flexibility or rigidity of the polyhedron.

For each of the four interior vertices $A_1$, $A_2$, $A_3$, $A_4$
consider the intersection of the cone of adjacent faces with a unit sphere
centered at the vertex.
This yields four spherical quadrilaterals $Q_i$ with side
lengths $\alpha_i, \beta_i, \gamma_i, \delta_i$ in this cyclic order.
Let $\phi$, $\psi_1$, $\theta$, and $\psi_2$ be the exterior dihedral angles
at edges $A_1 A_2,$ $A_2 A_3,$ $A_3 A_4$ and $A_4 A_1,$ respectively.
Clearly, it suffices to parameterize these dihedral angles
in order to describe a flexion of a Kokotsakis polyhedron. 

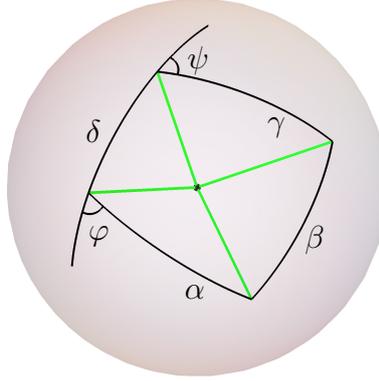
\begin{figure}[ht]
	\begin{center}
	\begin{picture}(150,150)
	\put(0,0){\includegraphics[scale=0.4]{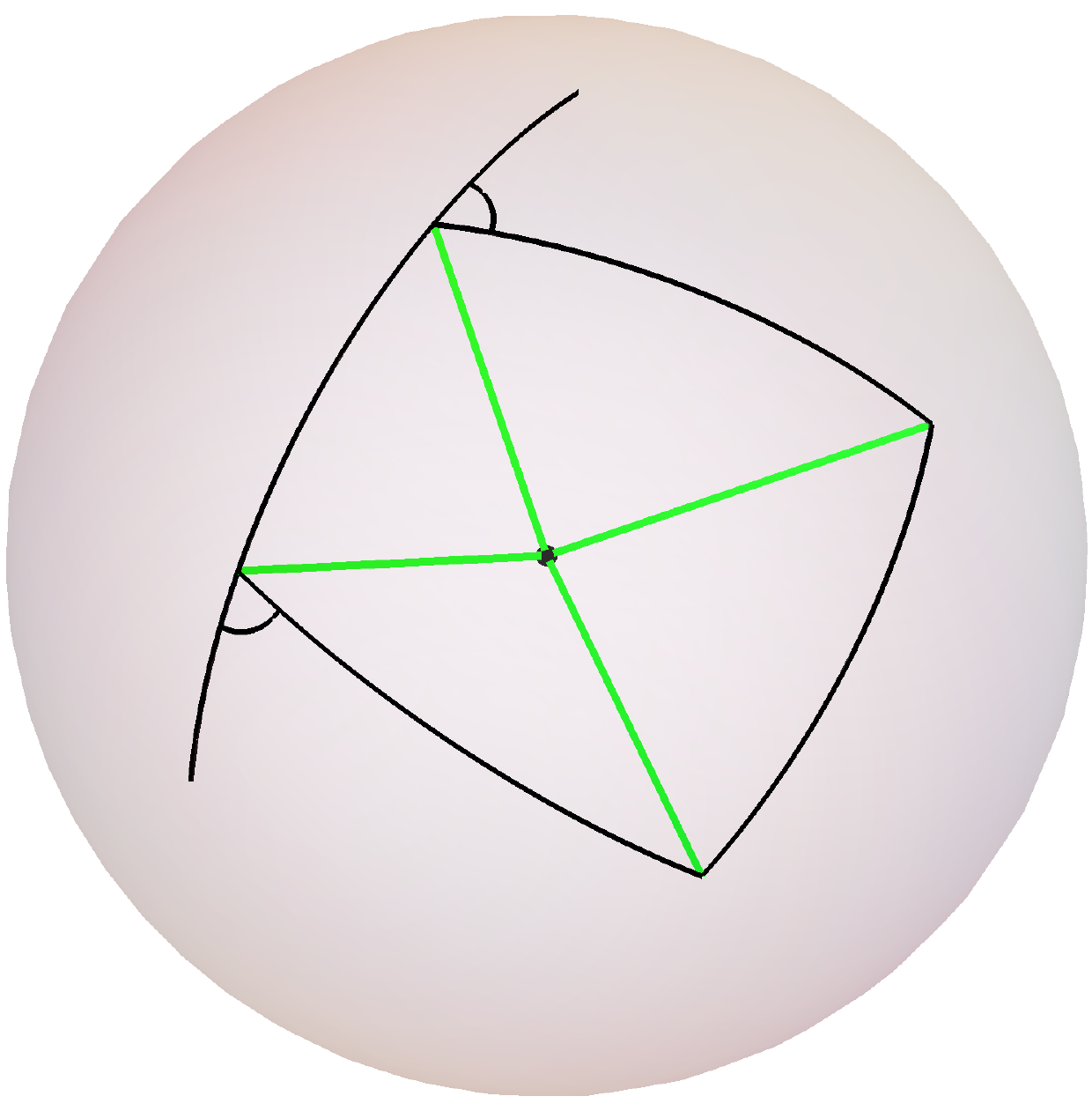}}
	\put(30,90){$\delta$}
	\put(98,94){$\gamma$}
	\put(112,50){$\beta$}
	\put(67,31){$\alpha$}
	\put(68,118){$\psi$}
	\put(31,53){$\phi$}
	\end{picture}
	\end{center}
	\caption{ Spherical quadrilateral $Q$}
	\label{fig:Quad_Notation}
\end{figure}

\subsection{Orthodiagonal quadrilaterals}
\label{sec:Orthog_Quad}

We call a spherical quadrilateral \emph{orthodiagonal} if its diagonals are orthogonal.
A spherical quadrilateral is called {\it elliptic}
if its planar angles satisfy 
$\alpha \pm \beta \pm \gamma \pm \delta \ne 0  \,(\mod 2\pi).$

Let $Q$ be a spherical elliptic orthodiagonal quadrilateral
with side lengths and angles  denoted as in \Href{Figure}{fig:Quad_Notation}. Then the orthodiagonality of $Q$ is equivalent to the following identity:
	\begin{equation} \label{eqn:orth_identity}
		\cos\alpha\cos\gamma = \cos\beta\cos\delta
	\end{equation}
(see  \Href{Lemma}{lem:OrthoRelation} below).

By definition put
	$z = \tan \frac{\phi}{2},$ $w = \tan \frac{\psi}{2}$. 
Then $(z, w)$ satisfy the following equation
(see Lemma 4.13 in \cite{izmestiev2016classification}):
	\begin{equation}
	\label{eqn:OrthoPolynomial}
		(z^2 + \lambda)(w^2 + \mu ) = \nu z w, \quad {z,w \in \mathbb{R}P^1 = \mathbb{R} \cup \infty,}
	\end{equation}
where $\lambda, \mu$ are the so-called {\it involution factors} defined as follows
\begin{gather}
	\label{eqn:def_la}
	\lambda \coloneqq
		\begin{cases}
		\frac{\tan\delta + \tan\alpha}{\tan\delta - \tan\alpha},
			&\text{if } \alpha \ne \frac{\pi}2 \text{ or } \delta \ne \frac{\pi}2,\\
		\frac{\cos\beta + \cos\gamma}{\cos\beta - \cos\gamma},
			&\text{if } \alpha = \delta = \frac{\pi}2;
		\end{cases}\\
	\label{eqn:def_mu}
	\mu \coloneqq
		\begin{cases}
		\frac{\tan\delta + \tan\gamma}{\tan\delta - \tan\gamma},
			&\text{if } \gamma \ne \frac{\pi}2 \text{ or } \delta \ne \frac{\pi}2,\\
		\frac{\cos\beta + \cos\alpha}{\cos\beta - \cos\alpha},
			&\text{if } \gamma = \delta = \frac{\pi}2;
		\end{cases}
\end{gather}
and $\nu$ is given by
\begin{equation}
	\label{eqn:def_nu}
	\nu \coloneqq
		\begin{cases}
		\frac{(\lambda - 1)(\mu - 1)}{\cos\delta},
			&\text{if } \delta \ne \frac\pi{2},\\
		2(\mu -1) \tan\alpha,
			&\text{if } \delta = \gamma = \frac\pi{2},\\
		2(\lambda - 1) \tan\gamma,
			&\text{if } \delta = \alpha = \frac\pi{2}.
		\end{cases}
\end{equation}

The involution factors $\lambda, \mu$  and $\nu$ are well-defined real numbers different from $0$.

\subsection{Orthodiagonal anti-involutive type}

A Kokotsakis polyhedron belongs to the {\it orthodiagonal anti-involutive type},
if its planar angles satisfy the following conditions. 

\begin{enumerate}
	\item[1)]
		All quadrilaterals $Q_i$ are orthodiagonal and elliptic:
		$$
			\cos\alpha_i \cos\gamma_i = \cos\beta_i \cos\delta_i, \qquad
			\alpha_i \pm \beta_i \pm \gamma_i \pm \delta_i \ne 0\ (\mod 2\pi).
		$$
	\item[2)]
		The involution factors at common vertices are opposite:
		\begin{equation}
			\label{eqn:LambdaEq}
			\lambda_1 = -\lambda_2, \qquad \mu_1 = -\mu_4, \qquad
			\mu_2 = -\mu_3, \qquad \lambda_3 = -\lambda_4.
		\end{equation}
	\item[3)] 
		The involution factors $\lambda_i, \mu_i$ and $\nu_i$ satisfy the following system:
		\begin{equation}
		\label{eqn:SystemProporResultants}
			\frac{\nu_1^2}{\lambda_1\mu_1} = \frac{\nu_3^2}{\lambda_3\mu_3}, \qquad
			\frac{\nu_2^2}{\lambda_2\mu_2} = \frac{\nu_4^2}{\lambda_4\mu_4}, \qquad
			\frac{\nu_1^2}{\lambda_1\mu_1} + \frac{\nu_2^2}{\lambda_2\mu_2} = 16.
		\end{equation}
\end{enumerate}

By \eqref{eqn:OrthoPolynomial} and \eqref{eqn:LambdaEq},
variables $(z, w_1, u, w_2)$ of a Kokotsakis polyhedron of the 
orthodiagonal anti-involutive type satisfy the system
\begin{equation}
	\label{eqn:main_system_with_la_mu_signed}
	\begin{aligned}
		P_1 = (z^2 + \lambda_1) (w_1^2 + \mu_1) -\nu_1 z w_1 =0, \\
		P_2 = (z^2 - \lambda_1) (w_2^2 - \mu_3) -\nu_2 z w_2 =0, \\
		P_3 = (u^2 + \lambda_3) (w_2^2 + \mu_3) -\nu_3 u w_2 =0, \\
		P_4 = (u^2 - \lambda_3) (w_1^2 - \mu_1) -\nu_4 u w_1 =0.
	\end{aligned}
\end{equation}
We are interested only in the {\it non-trivial} solutions of this system, that is,
 the one-parametric branches of the solution set chosen in a way  that each of $z, w_1, u, w_2$ is not a constant. 
\begin{rem}
	System \eqref{eqn:SystemProporResultants} expresses the proportionality
	of the resultants
	$R_{12} (w_1, w_2)=\mathop{\mathrm{res}}_z(P_1, P_2)$
	and $R_{34}(w_1, w_2) = \mathop{\mathrm{res}}_u(P_3, P_4)$.
\end{rem}

\begin{rem}
	There is an error in the last equation of \eqref{eqn:SystemProporResultants}
	in \cite{izmestiev2016classification}, Subsection 3.1.2.
\end{rem}

\subsection{Five-parametric system}
\label{subsec:five_parametric}

The problem of description of all flexible  Kokotsakis polyhedra
of the orthodiagonal anti-involutive type
is separated into two sub-problems in a natural way:
\begin{itemize}
	\item to describe the planar angles  $(\alpha_i, \beta_i, \gamma_i, \delta_i)$
	such that the corresponding polyhedron is of the orthodiagonal anti-involutive type;
	\item to describe its flexion, i.e. the trajectory in dihedral angles space.
\end{itemize} 
The latter  is equivalent to solving system \eqref{eqn:main_system_with_la_mu_signed},
which describes the configuration space of $(z, w_1, u, w_2)$,
with  given coefficients  $\lambda_i$, $\mu_i$, $\nu_i$ that satisfy \eqref{eqn:LambdaEq}
and \eqref{eqn:SystemProporResultants}.
In \Href{Theorem}{thm:real_conf_space} we give a complete parameterization
with one variable of a solution set of this system. 

It is more complicated to describe all admissible planar angles
$(\alpha_i, \beta_i, \gamma_i, \delta_i)$. 
It is natural to consider the angles
$(\delta_1, \delta_2, \delta_3, \delta_4)$ as parameters. 
Since $\delta_1 + \delta_2 + \delta_3 +\delta_4 = 2 \pi$,
it is a three parametric family.
From the algebraic point of view,
system \eqref{eqn:SystemProporResultants} together with identities \eqref{eqn:LambdaEq}
and definition \eqref{eqn:def_nu} is a system of three equations in four variables
$(\lambda_1, \mu_1, \lambda_3, \mu_3)$.
It is enough to introduce one more parameter
to describe a solution, which we do in \Href{Section}{sec:planar}.
On the other hand,  given $(\delta_1, \delta_2, \delta_3, \delta_4)$
and $(\lambda_1, \mu_1, \lambda_3, \mu_3)$ together with \eqref{eqn:LambdaEq}
one can recover angles $\alpha_i, \gamma_i$ from
\begin{equation} \label{eqn:a_g_of_la_mu}
	\tan\alpha_i = \frac{\lambda_i-1}{\lambda_i+1}\tan\delta_i, \quad
	\tan\gamma_i = \frac{\mu_i-1}{\mu_i+1}\tan\delta_i 
\end{equation}
in the case $\frac{\pi}{2} \notin \{\delta_1, \delta_2, \delta_3, \delta_4 \},$
and compute $\beta_i$ from
\begin{equation}\label{eqn:beta_cond}
	\cos\beta_i = \frac{\cos\alpha_i\cos\gamma_i}{\cos\delta_i},
\end{equation}
if the right-hand side is in $(-1,1)$.
That is, the angles of the base quadrilateral and $(\lambda_1, \mu_1, \lambda_3, \mu_3)$ determine all the needed data to construct a Kokotsakis polyhedron.

It is not obvious that such a geometric construction exists even one finds all the  planar angles of an OAI Kokotsakis polyhedron.
However, the following result gives the answer to this question.
\begin{dfn}\label{def:geom_req} 
We consider the following  {\it  geometric assumptions} on $(\delta_1, \delta_2, \delta_3, \delta_4)$ and
$(\lambda_i, \mu_i).$ 
\begin{enumerate}[label=\arabic*.]
	\item
		$(\delta_1, \delta_2, \delta_3, \delta_4)$ are the angles of a quadrilateral
		without right angles, that is $\delta_i > 0$ and $\delta_i \neq \frac{\pi}{2}$,
		where  $i=1,2,3,4$ and $\delta_1 + \delta_2 + \delta_3 +\delta_4 = 2 \pi.$
	\item
		$(\lambda_i, \mu_i)_1^4$  satisfy systems \eqref{eqn:LambdaEq},
		\eqref{eqn:SystemProporResultants} with $\nu_i$ given by \eqref{eqn:def_nu}.
	\item
		The angles $\{\alpha_i, \gamma_i, \beta_i \}_1^4$ given by
		\eqref{eqn:a_g_of_la_mu} and \eqref{eqn:beta_cond} are well-defined, that is,
		the right-hand side of \eqref{eqn:beta_cond} is in interval 
		$(-1,1).$ 
	\item
		Spherical quadrilaterals $Q_i$ with sides $(\alpha_i, \beta_i, \gamma_i, \delta_i)$,
		where  $1 \leq i \leq 4$, are elliptic.
\end{enumerate} 
\end{dfn}
\begin{thm}\label{thm:existence}
	Let $(\delta_1, \delta_2, \delta_3, \delta_4)$ and $(\lambda_i, \mu_i)_1^4$
	meet the geometric assumptions of \Href{Definition}{def:geom_req}.
	Then a Kokotsakis polyhedron with spherical quadrilaterals $Q_i$
	with sides $(\alpha_i, \beta_i, \gamma_i, \delta_i)$ exists,
	it is flexible and all its possible flexion are parameterized
	as in \Href{Theorem}{thm:real_conf_space}.  
\end{thm}
 
In \Href{Theorem}{thm:la-mu-parameterized} of \Href{Section}{sec:planar}
we parameterize all  $(\lambda_i, \mu_i)_1^4$ that
satisfy systems \eqref{eqn:LambdaEq}, \eqref{eqn:SystemProporResultants}
with $\nu_i$ given by \eqref{eqn:def_nu}. 

\begin{rem}
	We find it interesting that if $(\delta_1, \delta_2, \delta_3, \delta_4)$
	and $(\lambda_i, \mu_i)_1^4$  meet only first two of the geometric assumptions,
	a non-trivial solution of \eqref{eqn:main_system_with_la_mu_signed} still exists
	and is given by \Href{Theorem}{thm:real_conf_space}.
	However, the angles $\beta_i$ might fail to be determined by \eqref{eqn:beta_cond} if the absolute value of the right-hand side of \eqref{eqn:beta_cond} is bigger than 1.
	This follows from \Href{Lemma}{lem:proper_tau}
	and a condition of the existence of non-trivial solution
	in \Href{Theorem}{thm:real_conf_space}.
\end{rem}

\section{The real part of the configuration space}
\label{sec:conf_space}

The form of solution in real numbers of the system \eqref{eqn:main_system_with_la_mu_signed}
depends on the signs of $(\lambda_1, \mu_1, \lambda_3, \mu_3)$.
That is, it depends on  the signs of the entries in the matrix
\[
\begin{pmatrix}
	\lambda_1 & \mu_1 \\
	\lambda_2 & \mu_2 \\
	\lambda_3 & \mu_3 \\
	\lambda_4 & \mu_4
\end{pmatrix} = 
\begin{pmatrix}
	 \lambda_1  &  \mu_1 \\
	-\lambda_1  & -\mu_3 \\
	 \lambda_3  &  \mu_3 \\
	-\lambda_3  & -\mu_1
\end{pmatrix}.
\]
We call arrangement of signs in this matrix a {\it sign pattern}. 
It will be shown below in \Href{Lemma}{lem:sign_patterns}
that there is a unique possible arrangement of signs
up to re-enumeration of the vertices.
Using this observation, we assume that 
\begin{equation}
	\label{eqn:choosing_sign_pattern}
	\lambda_1, \mu_1 > 0\qquad \text{and} \qquad \lambda_3, \mu_3 < 0.
\end{equation}

\begin{thm}
\label{thm:real_conf_space}
Let  $\lambda_i, \mu_i, \nu_i$ satisfy systems \eqref{eqn:SystemProporResultants}, \eqref{eqn:LambdaEq}
	and assumption \eqref{eqn:choosing_sign_pattern}. 
	Then system \eqref{eqn:main_system_with_la_mu_signed} has a real non-trivial one-parametric set of solutions  if and only if
		$\zeta_1 \coloneqq \cfrac{|\nu_1|}{4\sqrt{\lambda_1\mu_1}}  > 1$.
The solution set  has four branches, which, up to the symmetry
	$(z, w_1, u, w_2) \to - (z, w_1, u, w_2)$,
	are given by
	\begin{equation}
	\label{eqn:sol_reduced_system}
		\begin{cases}
			z   =  \sgn\nu_1 \cdot \sqrt{\lambda_1} \cdot
				F(t) F(t + \frac{\pi}{2}), \\
			w_1 =  \sqrt{\mu_1} \cdot F(t) F(t - \frac{\pi}{2}), \\[0.8ex]
			u = \sgn\nu_4  \cdot \sqrt{-\lambda_3}\cdot
				\cfrac{\sqrt{\zeta_1+1} \pm \left(G(t) + G(t+\frac\pi2)\right)}
				{V(t) - V(t+\tfrac\pi2)},\\[1.5ex]
			w_2 =  \sqrt{-\mu_3} \cdot
				\cfrac{\sgn(\nu_1\nu_2)\sqrt{\zeta_1 +1} \pm \sgn(\nu_3\nu_4) \left(G(t)+G(t-\frac\pi2)\right)}
				{V(t) - V(t-\tfrac\pi2)}.
		\end{cases},
	\end{equation}
	where
	\begin{align}
		F(t) &=  \sin t\sqrt{\zeta_1 -1} + \sqrt{1+(\zeta_1 -1) \sin ^2 t}; \nonumber\\ 
		G(t) &= \sin t \sqrt{1+(\zeta_1 -1) \sin^2 t}; \nonumber \\ 
		V(t) &= \sin t \sqrt{1+(\zeta_1 -1) \cos^2 t}, \nonumber
	\end{align}
	$t \in [0, 2 \pi)$ is a parameter
	and the choice of $\pm$ in $u$ and $w_2$ is simultaneous.
\end{thm}

According to assumption \eqref{eqn:choosing_sign_pattern}
and by condition $\zeta_1 > 1$,
we see that all values under square roots in the previous formulas are positive.

\subsection{Uniqueness of the sign pattern}

Due to system \eqref{eqn:SystemProporResultants},
we have some restrictions to possible signs of $(\lambda_i, \mu_i)$. 
\begin{lem}\label{lem:sign_patterns}
	System (\ref{eqn:LambdaEq}-\ref{eqn:main_system_with_la_mu_signed})
	allows unique sign pattern of $(\lambda_i, \mu_i)$
	up to cyclic re-enumeration of vertices.
	The only possible sign pattern is
	\begin{equation}
	\label{eqn:sign_pattern}
	\mathrm{Sign}\ \begin{pmatrix}
		\lambda_1 & \mu_1 \\
		\lambda_2 & \mu_2 \\
		\lambda_3 & \mu_3 \\
		\lambda_4 & \mu_4
	\end{pmatrix} = 
		\begin{pmatrix}
		+ & +\\ - & +\\ - & -\\ + & -
		\end{pmatrix}.
	\end{equation}
\end{lem}

\begin{proof}
	We first show that for some vertex both $\lambda_i,\mu_i>0$.
	By the third equation in \eqref{eqn:SystemProporResultants},
	we see that either $\lambda_1\mu_1 > 0$ or $\lambda_2\mu_2>0$.
	If in this positive product both factors are also positive,
	we have found the sought vertex.
	If however both coefficients are negative,
	from the fact that in a sign pattern there are always
	exactly four pluses and four minuses (follows from \eqref{eqn:LambdaEq}),
	 we conclude that if in sign pattern there is
	a row of minuses, there should be a row of pluses.
	
	As cyclic permutations just change the order of  vertices of the base quadrilateral,
	we assume that $\lambda_1 > 0$ and $\mu_1 > 0$.
	In other words, known signs are:
	\[
		\mathrm{Sign}\ \begin{pmatrix}
			\lambda_1 & \mu_1 \\
			\lambda_2 & \mu_2 \\
			\lambda_3 & \mu_3 \\
			\lambda_4 & \mu_4
		\end{pmatrix} = 
		\begin{pmatrix}
		+ & +\\ - & *\\ * & *\\ * & -
		\end{pmatrix}
	\]

	We rewrite equation $P_1(z,w_1)=0$ of \eqref{eqn:main_system_with_la_mu_signed}
	in the following way:
	\[
		\left(\frac{z}{\sqrt{\lambda_1}}+\frac{\sqrt{\lambda_1}}{z}\right)
		\left(\frac{w_1}{\sqrt{\mu_1}}+\frac{\sqrt{\mu_1}}{w_1}\right) = 
		\frac{\nu_1}{\sqrt{\lambda_1\mu_1}}.
	\]
	Since absolute values of both parentheses are greater or equal to 2,
	for this equation  to have a one-parametric solution (and not discrete set of points),
	we obtain
	\[
		\frac{\nu_1^2}{\lambda_1\mu_1} > 16.
	\]
	By the third equation in \eqref{eqn:SystemProporResultants}
	\[
		\frac{\nu_2^2}{\lambda_2 \mu_2} = 16 - \frac{\nu_1^2}{\lambda_1\mu_1} < 0.
	\]
	Hence, $0 > \lambda_2 \mu_2 = \lambda_1 \mu_3$ or, equivalently, $\mu_3 < 0$.
	By the first equation in  \eqref{eqn:SystemProporResultants}, 
	we get that $\lambda_3 < 0$.
	Finally, the only possible sign pattern is \eqref{eqn:sign_pattern}.
\end{proof}

\subsection{Reduced system}

To simplify equations of system \eqref{eqn:main_system_with_la_mu_signed}
and \eqref{eqn:SystemProporResultants},
we scale the variables and consider a reduced system
(after the re-enumeration of vertices so that $\lambda_1,\mu_1>0$ and $\lambda_3,\mu_3<0$).
Making the following substitution in system \eqref{eqn:main_system_with_la_mu_signed}:
\begin{equation}
\label{eqn:sub_for_reduced_system}
	\begin{aligned}
		\xin_1  &= \xin_2 = \frac{z}{\sqrt{\lambda_1}} \sgn{\nu_1}, \qquad &
		\xin_3  &= \xin_4 = \frac{u}{\sqrt{-\lambda_3}} \sgn{\nu_4}, \\
		\etan_1 &= \etan_4 =  \frac{w_1}{\sqrt{\mu_1}},  \qquad &
		\etan_2 &= \etan_3 =  \frac{w_2}{\sqrt{-\mu_3}} \sgn(\nu_1\nu_2),\\
		\zeta_i &= \frac{|\nu_i|}{4\sqrt{|\lambda_i\mu_i|}} \quad
		\text{for} \quad  i=1,2,4, \quad&
		\zeta_3 &=  \frac{\nu_3}{4\sqrt{\lambda_3\mu_3}} \sgn(\nu_1\nu_2\nu_4);
	\end{aligned}
\end{equation}
we get 4 equations:

\begin{equation}
\label{eqn:reduced_system}
	\begin{cases}
		(\xin_1^2 + 1) (\etan_1^2 + 1)= 4 \zeta_1 \xin_1 \etan_1, \\
		(\xin_1^2 - 1) (\etan_3^2 + 1)= 4 \zeta_2 \xin_1 \etan_3, \\
		(\xin_3^2 - 1) (\etan_3^2 - 1)= 4 \zeta_3 \xin_3 \etan_3, \\
		(\xin_3^2 + 1) (\etan_1^2 - 1)= 4 \zeta_4 \xin_3 \etan_1,
	\end{cases}
\end{equation}
where $\zeta_1, \zeta_2$ and $\zeta_4$ are positive.

System \eqref{eqn:SystemProporResultants},
which describes conditions for the proportionality of the resultants,
can be rewritten in the form
\begin{equation}
\label{eqn:resultants_sys_zeta}
	\begin{cases}
		\zeta_1^2 = \zeta_3^2, \\
		\zeta_2^2 = \zeta_4^2, \\
		\zeta_1^2 - \zeta_2^2 = 1.
	\end{cases}
\end{equation}

\subsection{Equation with positive signature}
\label{sec:eq_positive_signs}

Let us study the first equation in \eqref{eqn:reduced_system}:
\begin{equation}\label{eqn:eq_positive_signs}
	{(\xin_1^2  + 1)} {(\etan_1^2  + 1)} = 4 \zeta_1 \xin_1 \etan_1.
\end{equation}
Recall $\zeta_1 > 0$.

One can see that $\xin_1$ and $\etan_1$ are of the same sign,
and there is a symmetry of the solution $(\xin_1, \etan_1) \to (-\xin_1, -\etan_1)$. 
We assume that $\xin_1, \etan_1 > 0$.
Now we consider new coordinates
$X = \log(\xin_1/\etan_1), Y = \log \xin_1\etan_1.$
In these coordinates, the equation takes the form:
\begin{equation} \label{eqn:cosh_eq_positive_signs}
	\cosh X + \cosh Y = 2 \zeta_1.
\end{equation} 
Thus, there is a one-parametric real solution set if and only if
\begin{equation}\label{eqn:inequality_zeta}
	\zeta_1 > 1,
\end{equation} 
which we have already shown in \Href{Lemma}{lem:sign_patterns}.

Solution to Equation \eqref{eqn:cosh_eq_positive_signs}
can be parameterized in an explicit way.
Let $\omega = \arcch \zeta_1$, $\varrho = \sqrt{2} \sh \frac{\omega}{2}$.
Then we can parameterize solution of \eqref{eqn:cosh_eq_positive_signs} as 
\[
	X = \arcsh (\varrho \cos t) \qquad \mbox{and} \qquad
	Y = \arcsh (\varrho \sin t),
\]
where $t \in [0, 2\pi)$, 
which, by direct calculations,
leads to the following representation of $\etan_1$ and $\xin_1$:
\[
	\xin_1 =   F(t) F(t + \frac{\pi}{2})
	\quad \mbox{and} \quad
	\etan_1 =  F(t) F(t - \frac{\pi}{2}),
\]
where
\[
F(t) = \sqrt{(\zeta_1-1) \sin ^2t+1}+\sqrt{\zeta_1-1} \sin t;
\]

Moreover,  the solution of \eqref{eqn:cosh_eq_positive_signs}
is bounded in the plane $(X,Y)$.
This means that $\xin_1$ and $\etan_1$ cannot be equal $0$ or $\pm \infty$.
Or, speaking geometrically, if equation \eqref{eqn:eq_positive_signs}
describes the configuration space of the spherical quadrilateral of a Kokotsakis polyhedron, 
then the corresponding side quadrilaterals never cross the base plane (see \Href{Lemma}{lem:flat_orthodiag}).

\subsection{Proof of  \Href{Theorem}{thm:real_conf_space}}

\begin{proof}
	By the results of \Href{Subsection}{sec:eq_positive_signs}
	and by \eqref{eqn:inequality_zeta} in particularly,
	we see that the real solution of the system is nonempty and non-trivial
	if and only if $\zeta_1 > 1$
	(recall that $\zeta_1, \zeta_2$ and $\zeta_4$ are positive).
	Also in \Href{Subsection}{sec:eq_positive_signs},
	the equation with positive sign signature was resolved,
	which leads to the explicit formulas for $z$ and $w_1$. 
	One can prove that the formulas for $u$ and $w_2$ from \eqref{eqn:sol_reduced_system}
	are the solution of the system by the direct substitution.
	However, we want to show that we have found all branches of the solution. 
	We can consider system \eqref{eqn:reduced_system} and variables
	$(\xin_1,\etan_1, \xin_3, \etan_3)$ instead of $(z, w_1, \theta, w_2)$.
	One can see that the solution of the first equation
	is given in \Href{Subsection}{sec:eq_positive_signs},
	the second equation is quadratic in $\etan_3$,
	the fourth equation is quadratic in $\xin_3$. 
	Hence, we have two possible branches of solution for both $\xin_3$ and $\etan_3$,
	and there are at most four branches of the solution up to the symmetry 
	$(\xin_1,\etan_1, \xin_3, \etan_3) \to (-\xin_1,-\etan_1, -\xin_3, -\etan_3)$. 
	So the only difficulty is to check that we have chosen the correct branches
	to meet the third equation of the system.
	We claim that the latter obstacle implies that there are
	exactly two branches of the solution with positive $\xin_1$ and $\etan_1$.
	Hence, we described  all possible solution with positive $\xin_1$ and $\etan_1$.
	Indeed, fixing $\xin_1$, we fix $(\etan_3^2+1)/\etan_3$.
	Therefore, fixing $\xin_1, \xin_3, \etan_1$, we fix $(\etan_3^2-1)/ \etan_3$.
	But $(\etan_3^2-1)/ \etan_3= (\etan_3^2+1)/ \etan_3 - 2/\etan_3$
	changes its value if we change the branch for $\etan_3$.
	Therefore, there is a unique branch of $\etan_3$ for each branch of $\xin_3$.
	By the symmetry
	$(\xin_1, \xin_3, \etan_1, \etan_3) \to (-\xin_1,-\etan_1, -\xin_3, -\etan_3)$,
	we describe all branches of the solution.
\end{proof}

\subsection{Reducible resultant}
As a linear substitution does not change the property of irreducabillity of the resultant, we consider the resultant $R^r_{12}$ of the first  two equations of reduced system 
\eqref{eqn:reduced_system} as polynomials in $\xin_1.$
$R^r_{12}$ is given by:
\begin{multline*}
\frac14 R^r_{12} = 1 + 2(1-2\zeta_1^2)\etan_1^2 + 2(1+2\zeta_2^2)\etan_3^2
	+ \etan_1^4 + 4(1-2\zeta_1^2+2\zeta_2^2)\etan_1^2\etan_3^2 + \etan_3^4 \\
	\strut + 2(1+2\zeta_2^2)\etan_1^4\etan_3^2 + 2(1-2\zeta_1^2)\etan_1^2\etan_3^4
	+ \etan_1^4 \etan_3^4.
\end{multline*}
When $\zeta_1^2-\zeta_2^2 = 1$ from \eqref{eqn:resultants_sys_zeta} holds true,
the resultant is the product of the following polynomials:
\[
	\frac14 R^r_{12} =
	\Big(\etan_1^2 \etan_3^2 + (\zeta_1+\zeta_2)^2(\etan_1^2-\etan_3^2) - 1\Big)
	\Big(\etan_1^2 \etan_3^2 + (\zeta_1-\zeta_2)^2(\etan_1^2-\etan_3^2) - 1\Big).
\]
This can be checked by a direct calculation.
Moreover, as we show  in  \Href{Lemma}{lem:resultant_complex} below, the resultant is reducible if and only if $\zeta_1^2-\zeta_2^2 = 1.$
\subsection{Symmetries}
Here we discuss some symmetries hidden in our equations. 

First of all, the symmetry $(z, w_1, u, w_2) \to  (-z, -w_1, -u, -w_2)$ or, equivalently, in terms of the dihedral angles 
$(\phi, \psi_1, \theta,  \psi_2) \to (-\phi, -\psi_1, -\theta,  -\psi_2),$ is just the symmetry with respect to a plane of the base quadrilateral. Clearly, it does not change a configuration. 

Secondly, our equations allow us to make the following substitution for planar angles at  vertex $A_i$ of the base quadrilateral:
\begin{enumerate}
\item  $(\alpha_i, \gamma_i) \to ( \alpha_i - \pi,  \gamma_i - \pi);$
\item  $(\alpha_i, \beta_i) \to ( \alpha_i - \pi,   \pi - \beta_i);$
\item $(\gamma_i, \beta_i) \to ( \gamma_i - \pi,   \pi - \beta_i).$
\end{enumerate}
Clearly, the corresponding spherical quadrilateral remains elliptic, and the involution factors do not change.  These symmetries may help with avoiding self-intersections of a polyhedron.
\subsection{Elliptic parameterization}
In \cite{izmestiev2016classification} the author parameterized a solution of \eqref{eqn:OrthoPolynomial} in terms of the Jacobi elliptic sines. Using the results of that paper, we provide  a sketch of the proof  of the following result in which we give an elliptic parameterization of the solution set of system 
\eqref{eqn:main_system_with_la_mu_signed}. A comprehensive description of the OAI flexible Kokotsakis polyhedra with the use of elliptic functions can be found in \cite{IzmFlexKokinPrep}.
\begin{thm}
\label{thm:real_conf_space_ellipttic}
Let  $\lambda_i, \mu_i, \nu_i$ satisfy systems \eqref{eqn:SystemProporResultants}, \eqref{eqn:LambdaEq}
	and assumption \eqref{eqn:choosing_sign_pattern}. 
	Then system \eqref{eqn:main_system_with_la_mu_signed} has a real non-trivial one-parametric set of solutions  if and only if
		$\zeta_1 \coloneqq \cfrac{|\nu_1|}{4\sqrt{\lambda_1\mu_1}}  > 1$.
The solution set  has four branches, which, up to the symmetry
	$(z, w_1, u, w_2) \to - (z, w_1, u, w_2)$,
	are given by
	\begin{equation}
	\label{eqn:sol_reduced_systemelllliptic}
		\begin{cases} 
			z   =  \sgn \nu_1 \cdot \sqrt{\lambda_1k} \sn{\left(K + i t \right)}, \\
			w_1 =  \sqrt{\mu_1k} \sn{\left(K + \frac{i K'}{2} + i t  \right)}, \\
			u = i \sgn \nu_4  \cdot \sqrt{-\lambda_3 k}
			\sn\Big(K + \frac{i K'}{2} \pm \sgn(\nu_1\nu_2) \left(K-\frac{i K'}{2} \right) + i t \Big),\\
			w_2 = i \sgn (\nu_1 \nu_2) \cdot \sqrt{- \mu_3 k},
		\sn\Big(K \mp \sgn(\nu_3\nu_4) \left( K-\frac{i K'}{2}\right) + i t  \Big) 
		\end{cases}
	\end{equation}
where  $k$ is the elliptic modulus given by
$k =  \left(\zeta_1  -  \sqrt{\zeta_1^2 - 1}\right)^2,$
$K$ and $\frac{i K'}{2}$ are the quarter periods of elliptic sine with modulus $k$;
$t \in [0, 2 K')$ is a parameter and the choice of $\pm$ and $\mp$ in $u$ and $w_2$ is simultaneous.
\end{thm}
\begin{proof}[Sketch of the proof.]
It was shown in \cite{izmestiev2016classification} that the solution set of 
\eqref{eqn:OrthoPolynomial} has  a parameterization of the form
\[
z = p \sn(t,k), \quad w = q \sn(t+\tau,k),
\]
where $\tau$  is a quarter-period of $\sn,$ that is, $\tau = n K + \frac{iK'}{2},$
and the amplitudes $p,q$ belong to $\mathbb{R}_+ \cup i\mathbb{R}_+.$
Moreover, there is a table with explicit formulas for $(\lambda, \mu, \nu)$ as functions of $(p,q,k)$ in Lemma 4.17 of \cite{izmestiev2016classification}:
\[
(\lambda, \mu, \nu) =
\begin{cases}
(\frac{p^2}{k}, \frac{q^2}k, \frac{2(1+k)}{k\sqrt{k}}pq), &\text{if } \tau = \frac{iK'}2\\
(\frac{p^2}{k}, \frac{q^2}k, -\frac{2(1+k)}{k\sqrt{k}}pq), &\text{if } \tau = 2K + \frac{iK'}2\\
(-\frac{p^2}{k}, -\frac{q^2}k, \frac{2i(1-k)}{k\sqrt{k}}pq), &\text{if } \tau = K + \frac{iK'}2\\
(-\frac{p^2}{k}, -\frac{q^2}k, -\frac{2i(1-k)}{k\sqrt{k}}pq), &\text{if } \tau = 3K + \frac{iK'}2
\end{cases}
\]
From  this table, we see that the elliptic modulus $k$ is a function of $\frac{\nu^2}{\lambda \mu}.$  Since $k \in (0,1)$  and by the choice of sign pattern \eqref{eqn:sign_pattern}, we get that the elliptic moduli of all four spherical quadrilaterals are the same in case of an OAI polyhedron. Then it is easy to recover what a quarter-period shift one can use in each equation of \eqref{eqn:reduced_system}.
\end{proof}

\begin{rem}
It follows that the OAI type of Kokotsakis polyhedra is a special class of the elliptic equimodular type in the notation of \cite{izmestiev2016classification}.
\end{rem}

We note  that the parameterizations from \Href{Theorem}{thm:real_conf_space_ellipttic} cannot be transformed to the parameterizations given by \Href{Theorem}{thm:real_conf_space}  using linear substitution of the parameter $t$ as the first one uses elliptic functions that are not elementary. However, under the substitution  $t \to - \frac{K'}{\pi} t,$
the parameterization given by \Href{Theorem}{thm:real_conf_space_ellipttic} is quite close to the parameterization given by \Href{Theorem}{thm:real_conf_space} for  $k$ small enough (see \Href{Figure}{fig:two_parametrization} and 
\Href{Figure}{fig:two_parametrization_diff}).

%
\section{Planar angles of  a flexible polyhedron}
\label{sec:planar}
%

In this section, we study $(\delta_1, \delta_2, \delta_3, \delta_4)$
and $(\lambda_i, \mu_i)$ that meet the first two of the geometric assumptions.
In the following theorem we manage to parameterize all such $(\lambda_i, \mu_i)$.
In fact, for solving  system \eqref{eqn:SystemProporResultants}, we need to make it more symmetrical
with a proper substitution, and we use new variables for this purpose.
We introduce these substitutions in the theorem and explain their sense later.

\begin{thm}\label{thm:la-mu-parameterized}
	Let all $\delta_i$ and $(\lambda_i, \mu_i)$
	meet the first two of the geometric assumptions. 
	Put
	\begin{equation}\label{sub:deltas_to_xyz}
		\s = \frac{\delta_1-\delta_2+\delta_3-\delta_4}{4}, \qquad
		x = \frac{\delta_1-\delta_3}{2}, \qquad
		y = \frac{\delta_2-\delta_4}{2}.
	\end{equation}
	Then there exists $\tau \in [0, 2\pi)$ and a proper choice of signs such that
	\begin{equation} \label{sub:la_mu_of_r_c}
		\lambda_i = \frac{1 \pm \sqrt{1-r_i^2} }{r_i} , \qquad
		\mu_i     = \frac{1 \pm \sqrt{1-c_i^2} }{c_i}  \qquad
			\text{for} \quad i=1,3,
	\end{equation}
	where $r_i, c_i$ are functions of $(\tau, x, y, \s)$ with the symmetries 
	\begin{align}
		\label{eqn:c_r_equations}
		\begin{aligned}
			c_1 (\tau, x, y, \s) &= r_1 (\tau + \pi, x, -y, \s), \\
			r_3 (\tau, x, y, \s) &= r_1 (\tau, -x, -y, \s),      \\
			c_3 (\tau, x, y, \s) &= r_1  (\tau + \pi, -x, y, \s).
		\end{aligned}
	\end{align}
	and $r_1$ is given by 
	\begin{equation}
		\label{eqn:r_1_parametric}
		r_1 (\tau, x, y, \s) =
			\frac{N + S\sqrt{L}}
			{2D},
	\end{equation}
	where
	\begin{align*}
		S  &= s_{10} \cos\tau  + s_{01} \sin\tau , \\
		L  &= l_{20}\cos^2 \tau + l_{11} \sin \tau \cos \tau +l_{02}\sin^2 \tau, \\
		N  &= n_{20} \cos^2 \tau  + n_{11} \sin \tau \cos \tau + n_{02}\sin^2 \tau, \\
		D  &= d_{20} \cos^2 \tau + d_{11} \sin\tau \cos\tau + d_{02} \sin^2 \tau,
	\end{align*}
	with
	\begin{align*}
		s_{10} &= \cos2x + \cos2y + 2\cos2\s,\\
		s_{01} &= \cos2x - \cos2y,\\
		l_{20} &=  (\cos2x - \cos2y)^2 + 8\cos2\s (\cos2x + \cos2y), \\
		l_{11} &= 2 (\cos 2x - \cos 2y) (\cos 2x + \cos 2y + 2\cos 2\s), \\
		l_{02} &= (\cos 2x + \cos 2y - 2\cos 2\s)^2,\\
		n_{20} &= \sin(x+y) \big(\sin(x-3y) + \sin(3x-y) 
			      + 6\sin(x-y+2\s) - 2\sin(x-y-2\s) \big), \\
		n_{11} &= 8 \left(\sin^2(x+\s) \cos^2(x-\s) + \sin^2(y-\s) \cos^2(y+\s)\right), \\
		n_{02} &= (\cos2y - \cos2x) (\cos2x + \cos2y - 2\cos2\s),\\
		d_{20} &= \cos(x-y) \big( \cos(3x+y) + \cos(x+3y) 
		          + 2\cos(x+y-2\s) + 4\cos(x+3\s) \cos(y-\s) \big), \\
		d_{11} &= \cos4x - \cos4y - 4 \sin(x+y) \sin(x-y+2\s), \\
		d_{02} &= \sin(x-y) \big( \sin(x+3y) - \sin(3x+y) 
		          + 2\sin(x+y-2\s) - 4\cos(x+3\s) \sin(y-\s)  \big).
	\end{align*}
\end{thm}

It is very intriguing that for each value of $\tau$ in this theorem such that all values under square roots are positive,
system \eqref{eqn:main_system_with_la_mu_signed} has real non-trivial solutions.

\begin{lem}
\label{lem:proper_tau}
	Let $\delta_i$ and $(\lambda_i, \mu_i)$
	meet the first two of the geometric assumptions.
	Then \eqref{eqn:sign_pattern} is the sign pattern of $(\lambda_i, \mu_i)$
	up to  re-enumeration of vertices \eqref{eqn:sub_for_reduced_system} and 
	$\frac{\nu_j^2}{\lambda_j\mu_j} > 16$ for $j$ such that $\lambda_j, \mu_j > 0$. 
\end{lem}

\subsection{A new symmetric system}

System \eqref{eqn:SystemProporResultants} can be rewritten in the following form:
\begin{equation}
	\label{eqn:res_system_with_lambdas}
	\begin{cases}
		\cfrac{(\lambda_1-1)^2(\mu_1-1)^2}{\lambda_1\mu_1\cos^2\delta_1} +
			\cfrac{(\lambda_2 - 1)^2(\mu_2-1)^2}{\lambda_2\mu_2\cos^2\delta_2} = 16,\\[10pt]
		\cfrac{(\lambda_1-1)^2(\mu_1-1)^2}{\lambda_1\mu_1\cos^2\delta_1} =
			\cfrac{(\lambda_3-1)^2(\mu_3 - 1)^2}{\lambda_3\mu_3\cos^2\delta_3},\\[10pt]
		\cfrac{(\lambda_2 - 1)^2(\mu_2-1)^2}{\lambda_2\mu_2\cos^2\delta_2} =
			\cfrac{(\lambda_4 - 1 )^2(\mu_4 -1)^2}{\lambda_4\mu_4\cos^2\delta_4}.
	\end{cases}
\end{equation}

Make a substitution
\begin{equation}
	\label{eqn:r_c_a_definitions}
	r_i =\frac{2\lambda_i}{\lambda_i^2+1},	\qquad
	c_i =\frac{2\mu_i}{\mu_{i}^2+1} 			\qquad 
	a_i = \frac{1}{\cos^2 \delta_i},
\end{equation}
and let 
\begin{equation}\label{eqn:Z_i_of_zeta}
 	Z_i=a_i(r_i^{-1}-1)(c_i^{-1}-1) = 4 \frac{\nu_i^2}{\lambda_i\mu_i}
 		= \frac{(\lambda_i-1)^2(\mu_i-1)^2}{4 \lambda_i\mu_i\cos^2\delta_i}.
\end{equation}

Then system \eqref{eqn:res_system_with_lambdas} is equivalent
to the system  of the following three equations 
$Z_1 + Z_2 = 4$, $Z_1 = Z_3$ and $Z_2 = Z_4$.
Adding a linear dependent equation,
system \eqref{eqn:res_system_with_lambdas}
can be rewritten in the following equivalent way 
\begin{equation}
\label{eqn:system_Z}
	\frac{1}{2} Z_i = 1 - (-1)^i  \tan \tau,
\end{equation}
where $\tau$ is parameter. 

System \eqref{eqn:LambdaEq} takes the form:
\begin{equation}
	\label{eqn:signs_r_c}
	r_1 = -r_2, \qquad c_1 = -c_4, \qquad c_2 = -c_3, \qquad r_3 = -r_4.
\end{equation}

The equations of system \eqref{eqn:system_Z} are linear in both $r_i$ or $c_i$.
Since, by the definitions, $\lambda_i$, $\mu_i$ are real and $\nu_i \neq 0$, 
we obtain a restriction on the values of  $r_i$, $c_i$: 
\begin{equation}
\label{eqn:restriction_R_C}
	-1 \leq  r_i,c_i < 1.
\end{equation}

It is clear that $(r_i, c_i)$ has the same sign pattern as $(\lambda_i, \mu_i)$
and that substitution \eqref{eqn:r_c_a_definitions} is inverse to \eqref{sub:la_mu_of_r_c}.

\subsection{Proof of \Href{Theorem}{thm:la-mu-parameterized}}

\begin{proof}
	Restriction \eqref{eqn:restriction_R_C} yields positivity 
	of radical expressions in \eqref{sub:la_mu_of_r_c}.
	So, it suffices to prove identities 
	\eqref{eqn:c_r_equations} and \eqref{eqn:r_1_parametric}.
	We start with the latter one.

	Using exclusion technique
	to find $r_1 (\tan \tau)$ from system \eqref{eqn:system_Z}
	is equivalent to finding roots of an equation of the second degree with coefficients determined by the coefficients of the system \eqref{eqn:system_Z}.
	Indeed, from the equation $Z_i = \operatorname{const}$,
	we see that functions $c_i (r_i)$ or $r_i (c_i)$ are   linear fractional transformations.
	Since composition of two linear fractional  transformations is a linear fractional  transformation,
	finding $c_1(r_1)$, $r_3 (c_1)$ and $c_2(r_1)$, $r_3 (c_2)$,
	we obtain two different M\"obius transformation functions for $r_3(r_1).$ 
	Therefore, general solution has a form: 
	\begin{equation}
		\label{eqn:r_1_tan_tau}
		r_1( \tan \tau) = \frac{P(\tan\tau)\pm\sqrt{Q( \tan\tau)}}{R( \tan \tau)},
	\end{equation} 
	where $P$ and $R$ are polynomials of 2nd degree and $Q$ is a polynomial of 4th degree. 
	However, we noticed that $Q$ always have a full square as a multiplier,
	when $\sum\delta_i=2\pi$.
	Using  our substitution \eqref{sub:deltas_to_xyz},
	we managed to factorize $Q$, merge two branches in \eqref{eqn:r_1_tan_tau}
	and  obtain \eqref{eqn:r_1_parametric}. 
	The explicit formulas for the coefficients were found and can be verified
	with the use of a system of computer algebra.
	We used Wolfram Mathematica 11 for this purpose.

	Since equation \eqref{eqn:r_1_tan_tau} is obtained by the exclusion technique,
	it describes all possible $r_1$.
	Note that it has two branches, however we have only one in \eqref{eqn:r_1_parametric}.
	So we need explain how we managed to merge the branches.
	The radical expression in \eqref{eqn:r_1_tan_tau} equals to  $\frac{S^2 L}{\cos^4 \tau}$.
	Since $S$ is the only function that have linear terms in $\cos \tau$ and $\sin \tau$,
	substitution $\tau \to \tau + \pi$ changes the sign of $S$
	and preserves the values of $N$, $L$, $D$ in \eqref{eqn:r_1_parametric},
	that is, it merges two branches of \eqref{eqn:r_1_tan_tau}.

	Let us explain identities \eqref{eqn:c_r_equations}.
	There are no more solutions as each equation of \eqref{eqn:system_Z}
	is linear in either $r_i$ or $c_i$.
	Apart from checking \eqref{eqn:c_r_equations} with direct calculations,
	we provide the following arguments.
	The dependencies on $x$, $y$, $\s$ arise from re-enumeration of vertices,
	Explanation of a ``phase shift'' of $\tau$ is trickier.
	By the symmetry of the system \eqref{eqn:system_Z},
	it is enough to explain a ``phase shift'' for $c_1$.
	Fixing $r_1$, we have exactly one possible value for $c_1$ by the equation 
	$Z_1 = 2 - 2 \tan \tau$ and two possible branches in \eqref{eqn:r_1_tan_tau}
	rewritten for $c_1$.
	The ``phase shift'' $\tau \to \tau + \pi$ corresponds to the right choice of the branch. 
\end{proof}

\subsection{Proof of \Href{Lemma}{lem:proper_tau} and \Href{Theorem}{thm:existence}}

\begin{proof}[Proof of \Href{Lemma}{lem:proper_tau}]
	By the first equation in \eqref{eqn:res_system_with_lambdas},
	either $r_1$, $c_1$ or $r_2$, $c_2$ are both non-zero and have the same sign. 
	Again, by Pigeonhole principle, it implies that there is a pair with positive numbers. 
	Without loss of generality, let it be pair $r_1$, $c_1$.
	By the second equation, the latter implies that $r_3$, $c_3$ has the same sign. 
	We claim that they are  negative, and, therefore,
	we have the case of the right sign pattern.

	Assume the contrary that $r_3, c_3 > 0.$ 
	By the definition of $Z_1$ and system \eqref{eqn:system_Z}, we have
	\[
		2 + 2\tan \tau  = Z_1 = \frac{1}{\cos^2 \delta_1}
		\left(\frac1{r_1} - 1\right) \left(\frac1{c_1} - 1 \right) > 0,   
	\]    
	or, equivalently, $2 >  - 2 \tan \tau$.
	Then, by the definition of $Z_2$, we have
	\[
		Z_2 = \frac{1}{\cos^2 \delta_2}(1/r_1 + 1) (1/c_3 + 1) >
			(1/r_1 + 1) (1/c_3 + 1) > 4 > 2  - 2 \tan \tau,
	\]
	which contradicts system \eqref{eqn:system_Z}.
	Therefore, $r_3, c_3 < 0$. 
	Then, by system \eqref{eqn:system_Z} and \eqref{eqn:Z_i_of_zeta},
	we obtain
	\[
		\frac{\nu_1^3}{\lambda_3\mu_3}= \frac{\nu_1^1}{\lambda_1\mu_1} =
		\frac{1}{4} Z_3 = \frac{1}{4 \cos^2 \delta_3}
			\left(\frac{1}{|r_3|} + 1\right)\left(\frac{1}{|c_3|} + 1\right) > 1.
	\]
\end{proof}

\Href{Theorem}{thm:existence} is a direct consequence of this lemma.

\begin{proof}[Proof of \Href{Theorem}{thm:existence}]
	To proof the theorem we need to check the following
	\begin{itemize}
	\item
		spherical quadrilaterals with sides
		$(\alpha_i, \beta_i, \gamma_i, \delta_i)$ exist;
	\item
		system \eqref{eqn:main_system_with_la_mu_signed} has non-trivial solution.
	\end{itemize}
	The first assumption follows from \Href{Lemma}{lem:existence_spherical_quad},
	the second is guaranteed by \Href{Lemma}{lem:proper_tau}.
\end{proof}

\section{Algorithm}\label{sec:alg}

In the following Theorem, we summarize all needed steps
to construct a flexible Kokotsakis polyhedron of the orthodiagonal anti-involutive type
and to describe its flexion using given values of the angles of the base quadrilateral. 
We are always looking for angles in interval $(0, \pi).$ 
However, one can allow to the angles $\alpha_i$ and $\gamma_i$
to be in $ (0, \pi) \cup (\pi, 2\pi).$
\begin{thm}\label{thm:algorithm}
	Given parameters $(\delta_1, \delta_2, \delta_3, \delta_4)$,
	the algorithm of constructing a flexible Kokotsakis polyhedron
	of the orthodiagonal anti-involutive type is the following:
	\begin{enumerate}
	\item
		Check whether $(\delta_1, \delta_2, \delta_3, \delta_4)$
		meet the first of the geometric assumptions, that is,
		they form a quadrilateral without right angles.
	\item
		Calculate $(x,y,\s)$ using substitution \eqref{sub:deltas_to_xyz}.
	\item
		Check whether there exists $\tau$ such that $r_1$, $r_3$, $c_1$, $c_3$
		given by \eqref{eqn:r_1_parametric} and \eqref{eqn:c_r_equations}
		satisfy inequality \eqref{eqn:restriction_R_C} (see \Href{Subsection}{subsec:screening} for how this can be done).
		Set $r_2 = - r_1$, $r_4 = - r_3$, $c_2 = -c_3$, $c_4 = -c_1$.
	\item
		Calculate the angles $\alpha_i$ and $\gamma_i$ from  
		\[
			\tan \alpha_i = \sigma^\alpha_i \sqrt{\frac{1- r_i}{1+r_i}} \tan \delta_i
				\quad \text{and} \quad 
			\tan \gamma_i = \sigma^\gamma_i \sqrt{\frac{1- c_i}{1+c_i}} \tan \delta_i,
		\]
		where $\sigma^\bullet_i=\pm1$ are chosen in a way to satisfy:
		$\sigma^\alpha_1\sigma^\alpha_2=
		\sigma^\alpha_3\sigma^\alpha_4=
		\sigma^\gamma_1\sigma^\gamma_4=
		\sigma^\gamma_2\sigma^\gamma_3=1$.
		Select angles to be from $0$ to $\pi$.
	\item
		Check that  $\frac{\cos \alpha_i \cos \gamma_i}{\cos \delta_i} \in (-1,1)$. 
		Find $\beta_i$ using \eqref{eqn:beta_cond}.
	\item
		Check that obtained spherical quadrilaterals $Q_i$ with sides 
		$(\alpha_i, \beta_i, \gamma_i, \delta_i)$ are elliptic.
	\item
		Find $\lambda_i, \mu_i, \nu_i$ using formulas
		\eqref{eqn:def_la}, \eqref{eqn:def_mu} and \eqref{eqn:def_nu}.
	\item
		Re-enumerate vertices to get \eqref{eqn:sign_pattern} as a sign pattern of the system. 
	\item
		Find $(z, w_1, u, w_2)$
		using formulas of \Href{Theorem}{thm:real_conf_space}.
		Calculate dihedral angles using 
		\[
			(\phi, \psi_1, \theta, \psi_2) =
				(2\arctan z, 2\arctan w_1, 2\arctan u, 2\arctan w_2).
		\]
	\end{enumerate}
\end{thm}
\begin{proof}
	If all requirements are fulfilled,
	an existence and flexibility of a Kokotsakis polyhedron
	constructing using this algorithm follows from
	\Href{Theorem}{thm:existence} and \Href{Theorem}{thm:real_conf_space}. 
\end{proof}

\begin{rem} 
	We do not check whether the obtained surface is not self-intersecting.
	There remains a possibility that at every moment during flexion,
	there are two intersecting facets.
\end{rem}

 It's possible to derive a implicit formula for a set of $(x,y,\s)$
 (or, equivalently, $(\delta_1, \delta_2, \delta_3, \delta_4)$)
 for which there exists $\tau$ such that $-1<r_i,c_i<1$
 and the angles $\beta_i$ are well-defined.
 However, the resulting formula is likely to be enormous and incomprehensive.
 We instead explain how to check the conditions and
 show the result of numerical screening in the next subsection.    

\subsection{Screening of the space of parameters $(\delta_1, \delta_2, \delta_3, \delta_4)$}\label{subsec:screening}

\begin{figure}[t]
	\includegraphics[scale=1.1]{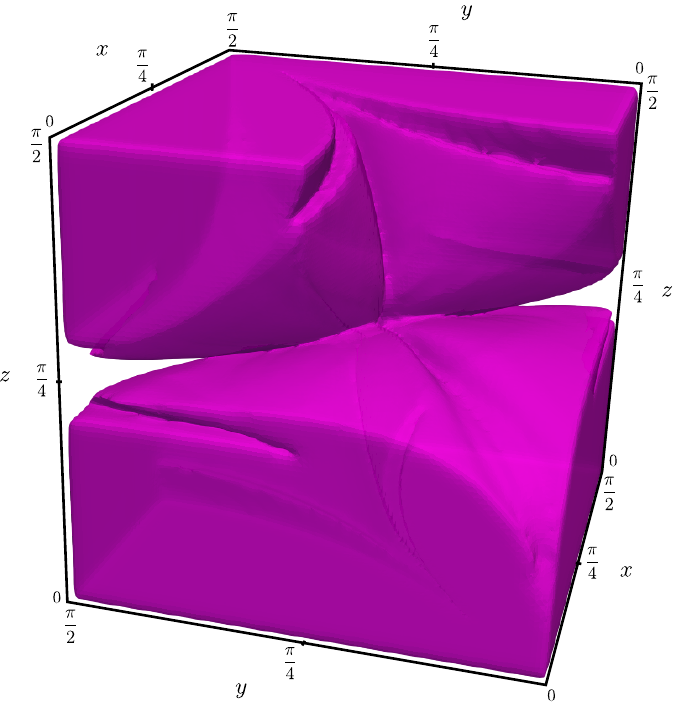}
	\caption{
	    \label{fig:screening result}
		Screening result.
		Due to the symmetry of the system, the set is centrally symmetric
		with respect to $\left(\frac{\pi}{4}, \frac{\pi}{4}, \frac{\pi}{4}\right)$
		and is symmetric with respect to the plane $x=y.$}
\end{figure}

In \Href{Theorem}{thm:algorithm}, one need to find a proper parameter $\tau$
for a given angles $(x, y, \s)$
(or, equivalently, $(\delta_1, \delta_2, \delta_3, \delta_4)$).
We run a computer screening and found numerically the set of $(x, y ,\s)$
for which there exists $\tau$ such that
$(\lambda_i, \mu_i)$ are given by \Href{Theorem}{thm:la-mu-parameterized}
and $\{\delta_1, \delta_2, \delta_3, \delta_4\}$ meet the geometric assumptions.
The result of our numerical computation is in \Href{Figure}{fig:screening result},
here we do not show surfaces on which at least one $Q_i$ is not elliptic.
\Href{Figure}{fig:screening result deltas} shows the result of screening
in $(\delta_1,\delta_3,\delta_2)$ coordinates for convex and non-convex
base quadrialareals.

We claim that it can be done with an arbitrary precision or, probably,
explicit formulas for the boundary can be found,
as all the restrictions on $\tan \tau$ are polynomial.

\begin{figure}[t]
	\includegraphics[scale=1]{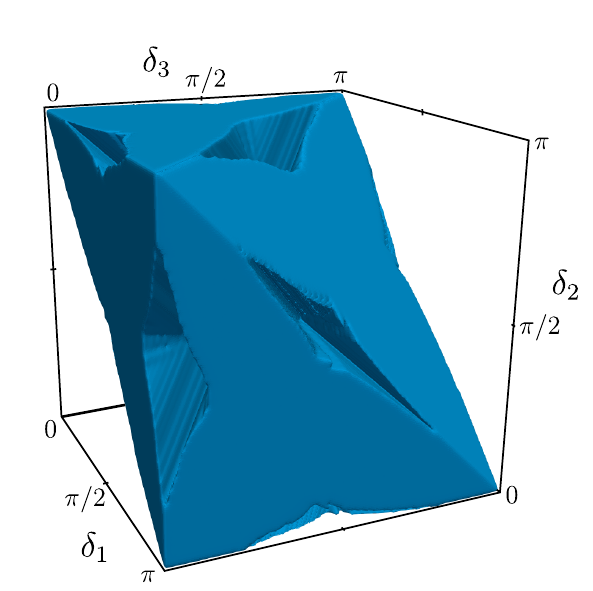}
	\includegraphics[scale=1]{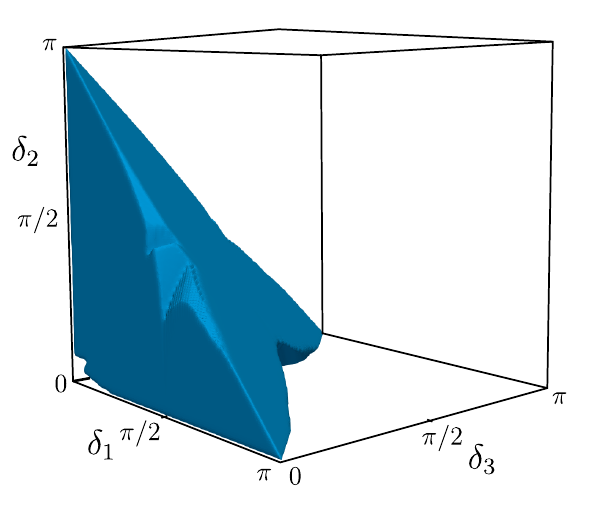}
	\caption{
	    \label{fig:screening result deltas}
		Set of $(\delta_1, \delta_2, \delta_3)$ with existing
		flexible an OAI polyhedron. On the left: with convex base quadrilateral;
		on the right: with non-convex base quadrilateral 
		($\delta_4>\pi$).
		}
\end{figure}

More precisely, 
\begin{enumerate}
	\item
		$\frac{L(\tau, x, y, z)}{ \cos^2 \tau} \geq 0$ is clearly a polynomial in $\tan \tau$.
	\item
		$r_i <  C = \operatorname{const}$ (and others comparisons for $r_i, c_i$ with constant $C$) 
		is equivalent to a proper system of the following  polynomial inequalities 
		\[
			\left(\frac{2C D  - N}{S \cos \tau}\right)^2 \geq \frac{L}{\cos^2 \tau},\qquad
			\left(\frac{2C D  - N}{S \cos \tau}\right)^2 \leq \frac{L}{\cos^2 \tau},
		\]
 		and comparison of $\frac{2C D  - N}{S \cos \tau}$,
 		$\frac{D}{\cos^2 \tau}$ and $\frac{S}{\cos \tau}$ with zero.
	\item
		$\beta_i$ is well-defined if and only if
		\[
			\frac{\cos^2 \delta_i}{\cos^2 \alpha_i \cos^2 \gamma_i} \leq 1.
		\]
		By  $\frac{1}{\cos^2 \alpha} = 1 + \tan^2 \alpha,$ the latter is
		\[
			\cos^2 \delta_i \left(\frac{1 - r_i}{1 + r_i} \tan^2 \delta_i  +1\right)
			\left(\frac{1 - c_i}{1 + c_i} \tan^2 \delta_i  +1\right) \geq 1,
		\] 
		which is equivalent to a system of polynomial inequalities.
\end{enumerate}

%
\section{Geometric properties}
\label{sec:geom_pr}
%

\subsection{The configuration space of a spherical quadrilateral}
Side lengths and angles  of a spherical quadrilateral $Q$ are denoted
in \Href{Figure}{fig:Quad_Notation_with_chi}, and $z = \tan \frac{\phi}{2},$ $w = \tan \frac{\psi}{2}$.

\begin{figure}[h]
	\begin{center}
	\begin{picture}(150,150)
	\put(0,-15){\includegraphics[scale=0.6]{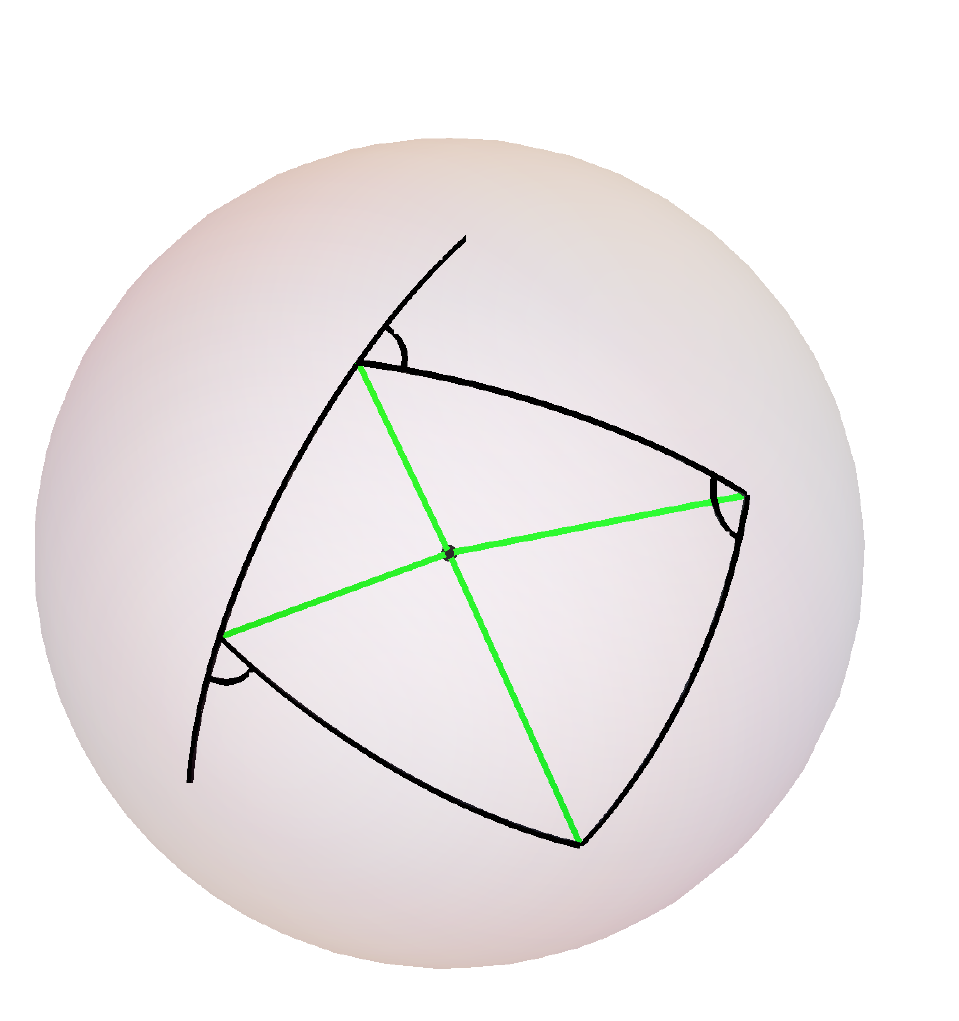}}
	\put(40,75){$\delta$}
	\put(98,94){$\gamma$}
	\put(122,40){$\beta$}
	\put(70,17){$\alpha$}
	\put(71,103){$\psi$}
	\put(37,35){$\phi$}
	\put(115,71){$\chi$}
	\end{picture}
	\end{center}
	\caption{ Spherical quadrilateral $Q$}
	\label{fig:Quad_Notation_with_chi}
\end{figure}
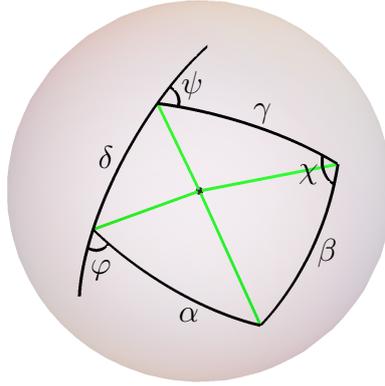

As was discovered by Bricard  \cite{bricard1897memoire},
the equation of the configuration space of $Q$ has the form 
    \begin{equation}
        \label{eqn:EulChasles}
        P(z,w) =  a_{22}z^2w^2 + a_{20}z^2 + a_{02}w^2 + 2a_{11}zw + a_{00} = 0,
    \end{equation}
where the coefficients are trigonometric functions of the angles of $Q.$

By allowing $z$ to take complex values, we arrive at a complex algebraic curve
    \begin{equation}
        \label{eqn:ConfSpace}
        C = \{(z, w)\in \mathbb{C}P^1 \times \mathbb{C}P^1  \mid P(z, w) = 0\}
    \end{equation}
with two coordinate projections
    \begin{equation}
        \label{eqn:CoordProj}
        P_z: C \rightarrow \CP^1 \ni z \quad
        \text{and} \quad P_w: C \rightarrow \CP^1 \ni w.
    \end{equation}
If $Q$ is of  elliptic type then $C$ is  an elliptic curve. 
The projections in \eqref{eqn:CoordProj} are
\emph{two-fold branched covers with exactly 4  points in the branch set}.
Since equation \eqref{eqn:EulChasles} is  quadratic in one variable,
one can easily find explicit formulas for points of
the branch set by solving a quadratic equation.
The complete classification of different classes of spherical quadrilaterals
and their branch sets can be found in \cite{izmestiev2016classification}
Subsection 2.4 and Lemma 4.10.

Denote by
    \begin{align*}
        i \colon C &\to C \qquad &\qquad j \colon C &\to C\\
        (z,w) &\mapsto (z,w') &\qquad (z,w) &\mapsto (z',w)
    \end{align*}
the deck transformations of $P_w$ and $P_z$. 
If $z$ and $w$ are real, we can  give a geometric interpretation of these {\it involutions}: $i$ and $j$ act by folding the quadrilateral along one of its diagonals,
see \Href{Figure}{fig:Involutions}.
This means that realizable 
branch points correspond to a degenerate case when quadrilateral becomes  a triangle.
Hence, we conclude

\begin{lem}
    \label{lem:InvolutionBranchedPoint}
    Let $Q$ be a spherical quadrilateral of  elliptic type. We use the notation as in \Href{Figure}{fig:Quad_Notation_with_chi} and put $z = \tan \frac{\phi}{2}.$
    Then $z$  belongs to the branch set
    of $P_z$ if and only if $\chi = 0 (\mathrm{mod}\, \pi).$
\end{lem}

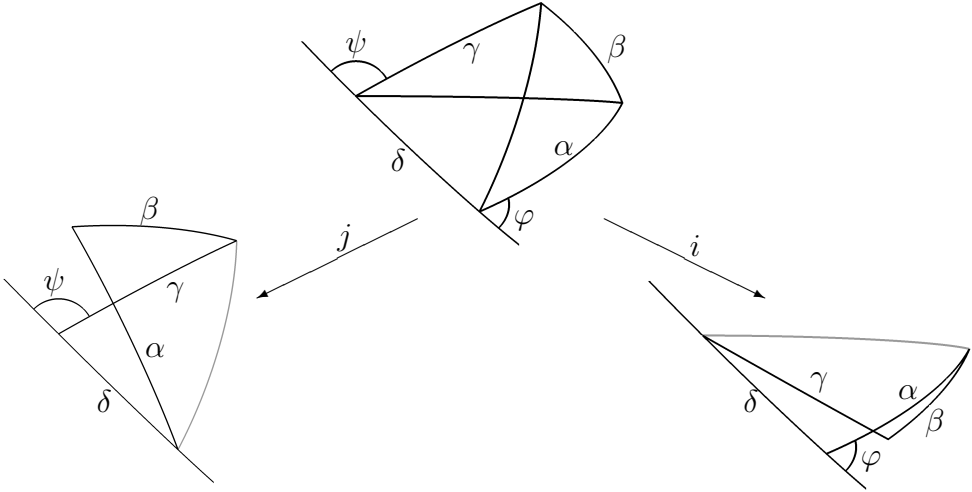
\begin{figure}[ht]
    \begin{center}
        \begin{picture}(240,200)
        \put(40,90){\includegraphics[scale=0.4]{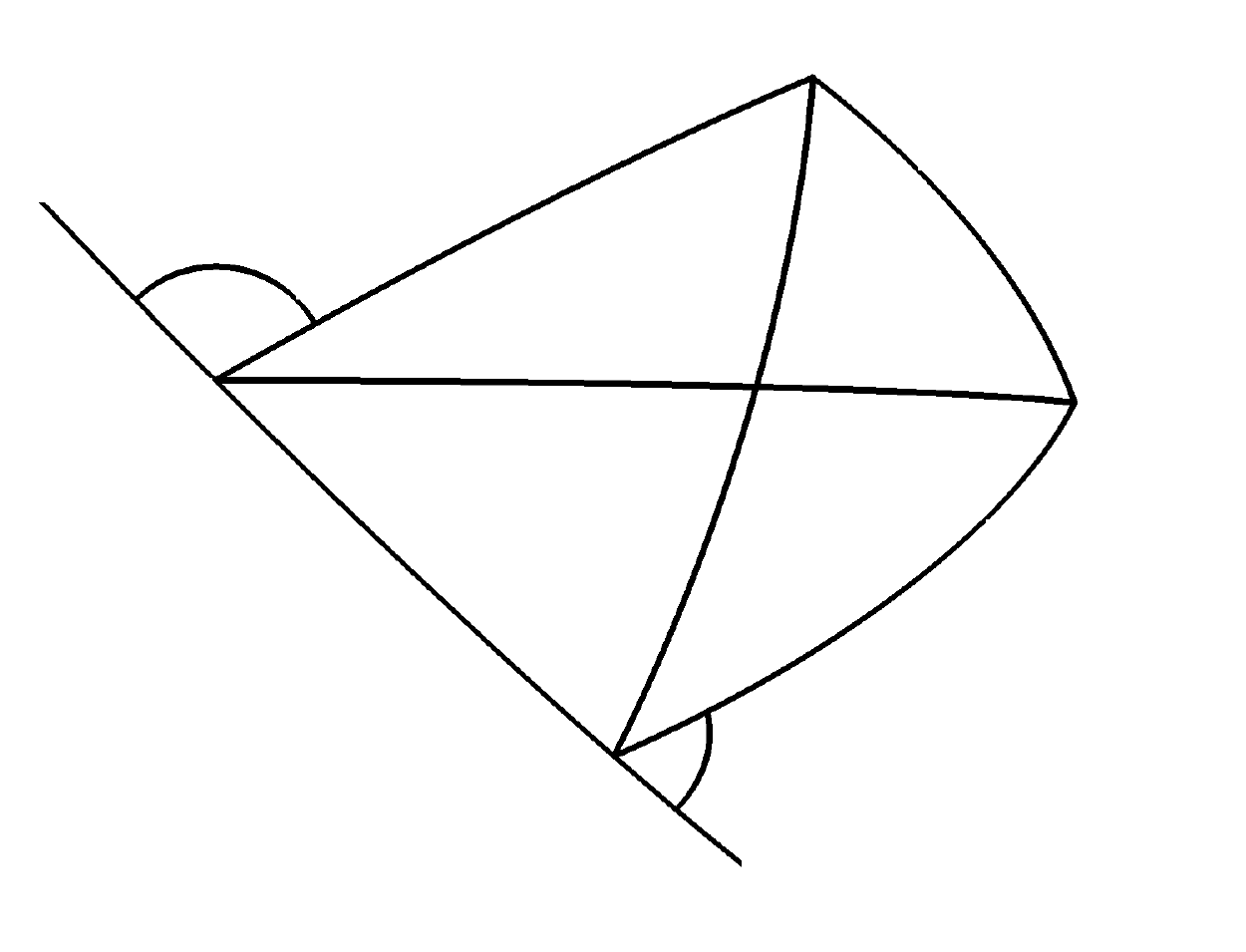}}
        \put(-70,-0){\includegraphics[scale=0.273]{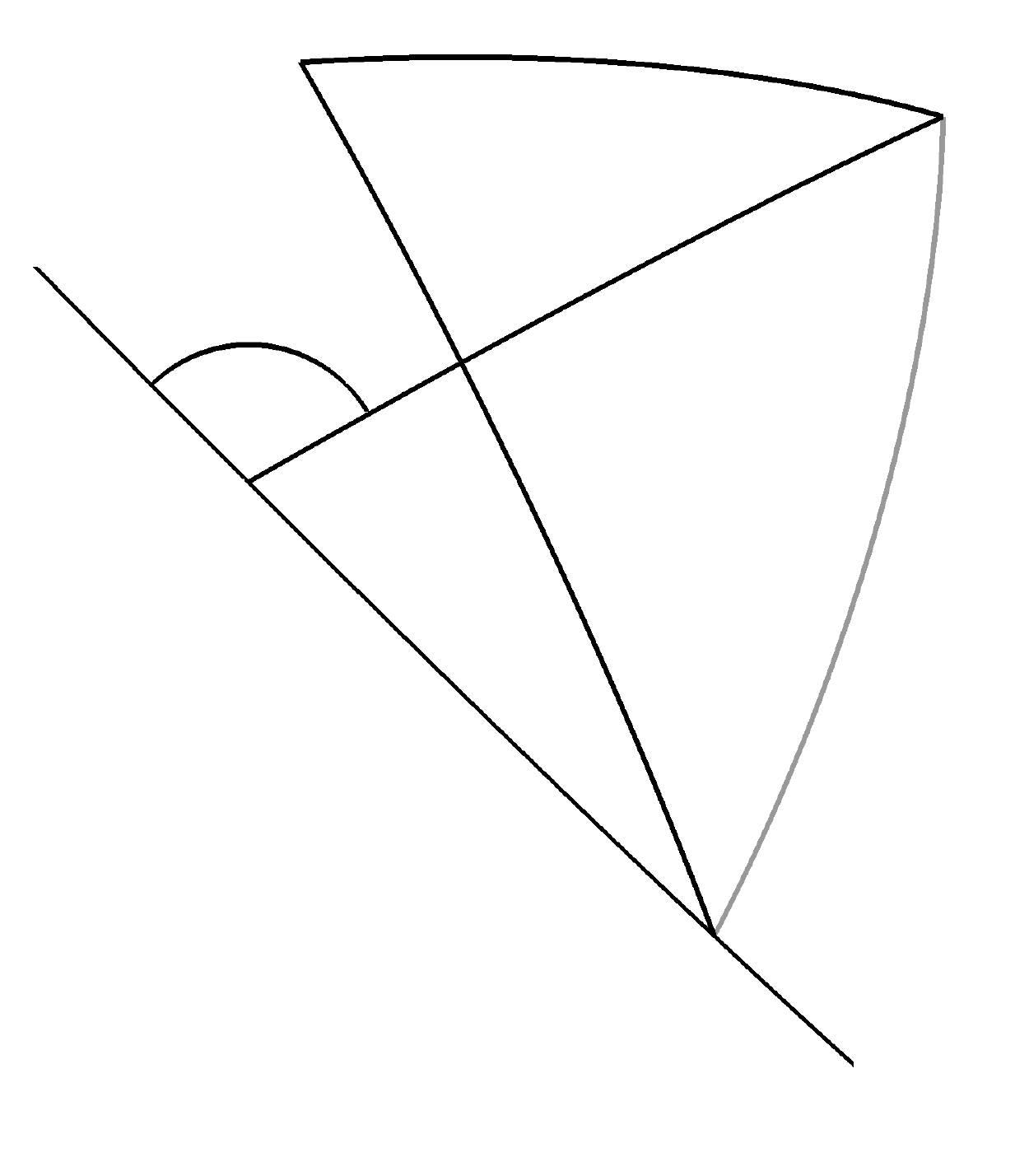}}
        \put(170,-0){\includegraphics[scale=0.4]{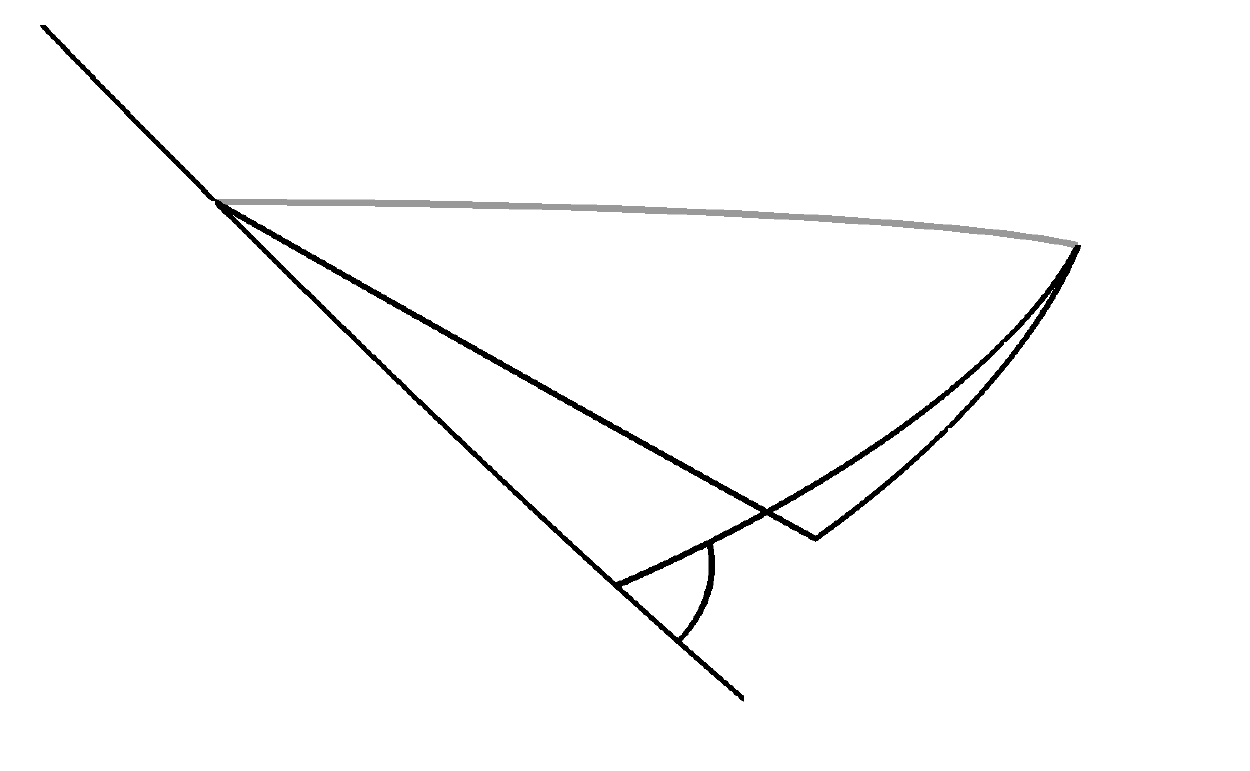}}
        \put(78,128){$\delta$}
        \put(-32,37){$\delta$}
        \put(210,37){$\delta$}
        \put(61, 173){$\psi$}
        \put(-52, 83){$\psi$}
        \put(124, 108){$\phi$}
        \put(254, 17){$\phi$}
        \put(105,171){$\gamma$}
        \put(-6,81){$\gamma$}
        \put(235,47){$\gamma$}
        \put(159,171){$\beta$}
        \put(-16,110){$\beta$}
        \put(278,30){$\beta$}
        \put(139,134){$\alpha$}
        \put(268, 42){$\alpha$}
        \put(-14, 58){$\alpha$}
        \put(158,110){\vector(2,-1){60}}
        \put(88,110){\vector(-2,-1){60}}
        \put(190,96){$i$}
        \put(58,100){$j$}
        \end{picture}
    \end{center}
    \caption{
        \label{fig:Involutions}
        Involutions $i$ and $j$ on the configuration space of a quadrilateral.
    }
\end{figure}

\subsection{Scissors-like linkage}
Two adjacent vertices  of the base quadrilateral share the common dihedral angles,
which means the corresponding  spherical quadrilaterals  are
{\it coupled} by  means of the angle (see \Href{Figure}{fig:Scissors}).
\begin{figure}
    \begin{center}
    \begin{picture}(110,100)
    \put(-50,0){\includegraphics[scale=0.6]{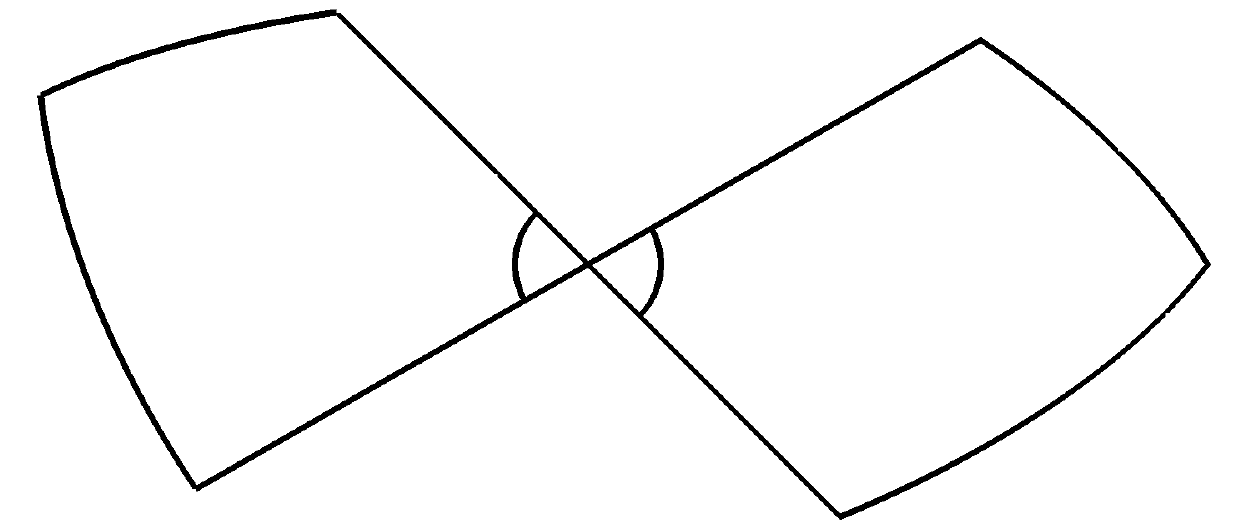}}
    \put(80,71){$\delta_1$}
    \put(140,71){$\gamma_1$}
    \put(64,17){$\alpha_1$}
    \put(130,13){$\beta_1$}
    \put(30,71){$\delta_2$}
    \put(-25,89){$\gamma_2$}
    \put(12,17){$\alpha_2$}
    \put(-48,40){$\beta_2$}
    \end{picture}
    \end{center}
    \caption{
        \label{fig:Scissors}
        Two coupled spherical quadrilaterals associated
        with the edge $A_1 A_2$. The two marked
        angles  are required to stay equal during the deformation.
    }
\end{figure}
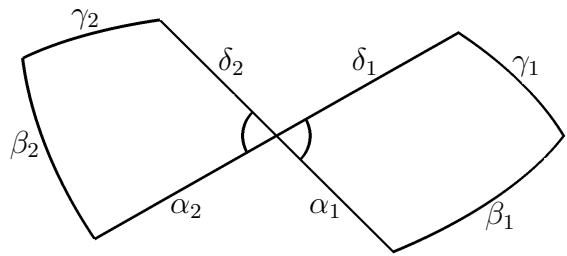
There are four such {\it scissors-like linkages} corresponding to edges 
$A_1A_2$, $A_2A_3$, $A_3A_4$ and $A_4A_1$
of the base quadrilateral in a Kokotsakis polyhedron.
Isometric deformations of the polyhedron correspond to motions of these linkages. 
Algebraically speaking, these linkages are described
by the rows and columns of system \eqref{eqn:PolSystem},
and the system itself describes possible dihedral angles of the polyhedron.
Moreover, a Kokotsakis polyhedron is flexible if and only if the system of polynomial
equations \eqref{eqn:PolSystem} has a one-parameter family of solutions over the reals
(\cite{izmestiev2016classification}, Lemma 2.2).
One possible approach to find a solution is as follows:
\begin{enumerate}
    \item
        consider a pair of opposite edges $A_1 A_2$ and $A_3 A_4$
        together with corresponding linkages,
        which are described by systems $\{ P_1 = 0, P_2 = 0\}$ and 
        $\{P_3 = 0, P_4 = 0\},$ respectively;
    \item
        exclude common variables $z$ and $u$ in the corresponding systems
        by computing  the resultants
	$R_{12} (w_1, w_2)=\mathop{\mathrm{res}}_z(P_1, P_2)$
	and $R_{34}(w_1, w_2) = \mathop{\mathrm{res}}_u(P_3, P_4)$.
\end{enumerate}
The polyhedron is flexible if and only if the algebraic sets $R_{12} = 0$ and $R_{34} = 0$
have a common irreducible component.
This means that they are reducible or irreducible simultaneously.
Stachel and Nawratil described all flexible classes for reducible $R_{12}$ and $R_{34}$.
\begin{rem}
    Geometrically speaking, the zero set of the resultant $R_{12} (w_1, w_2)$
    gives an implicit dependence of the non-common angles
    of spherical quadrilaterals $Q_1$ and $Q_2$,
    and $R_{34}$ does the same for $Q_3$ and $Q_4$.
    The polyhedron is flexible if and only if these ``dependencies'' have a common branch,
    that is, on can join the scissors-like linkages $(Q_1, Q_2)$ and $(Q_3, Q_4)$
    during the corresponding flexion.
\end{rem}

Izmestiev \cite{izmestiev2016classification} considers the complexified configuration space of a spherical linkage
(that is, the set of all complex solutions of the corresponding polynomial system).
The following lemma is a direct consequence  of the results of  \cite{izmestiev2016classification}.  It follows from Lemma 5.1 and assertion (4) of Lemma 4.10 of that paper.
\begin{lem}
    \label{lem:resultant}
   Let  polynomial system $\{ P_1 (z, w_1) = 0 = P_2 (z, w_2)\}$ 
    describe a scissors-like linkage of two elliptic spherical quadrilaterals $Q_1$ and 
    $Q_2.$  The resultant $R_{12} (w_1, w_2)$ is reducible if and only if the branch sets of $w_1$ and $w_2$ coincide.
\end{lem}
\begin{rem}
We  note that Izmestiev understands the property of irreducibility of a coupling of two spherical quadrilaterals as  the irreducibility of an algebraic set that describes the  corresponding complexified configuration spaces.  One should not confuse the properties of irreducibility of the resultant with the irreducibility of the coupling. 
\end{rem} 
As was mentioned above, Izmestiev showed that there is only one possible candidate 
for a flexible Kokotsakis polyhedron with  irreducible resultant $R_{12}$. 
In particular, all scissors-like linkages of such a polyhedron must be of a special type, which the  author called  anti-involutive coupling.
 
\subsection{Anti-involutive coupling and its properties}
A coupling of two orthodiagonal quadrilaterals is called {\it anti-involutive}
if their involution factors at the common vertex are opposite,
e.g. $\lambda_1 = - \lambda_2$ for two coupled spherical quadrilaterals associated with the edge $A_1 A_2$.

\begin{lem}
    \label{lem:linkage_prop}
    Let spherical orthodiagonal quadrilaterals $Q_1$ and $Q_2$ form
    a scissor-like linkage as in \Href{Figure}{fig:XiAngles}
    with the involution factors $\lambda_1$ and $\lambda_2$ in the common vertex.
    Then
    \begin{align*}
        \lambda_1 =  \lambda_2 &\Leftrightarrow \xi_1 = \xi_2, \xi'_1 = \xi'_2 \text{ during flexion}\\
        \lambda_1 = -\lambda_2 &\Leftrightarrow \xi_1 = \xi'_2, \xi'_1 = \xi_2 \text{ during flexion}
    \end{align*}
\end{lem}

\begin{figure}[ht]
    \begin{center}
        \begin{picture}(250,140)
        \put(0,0){\includegraphics[scale=0.7]{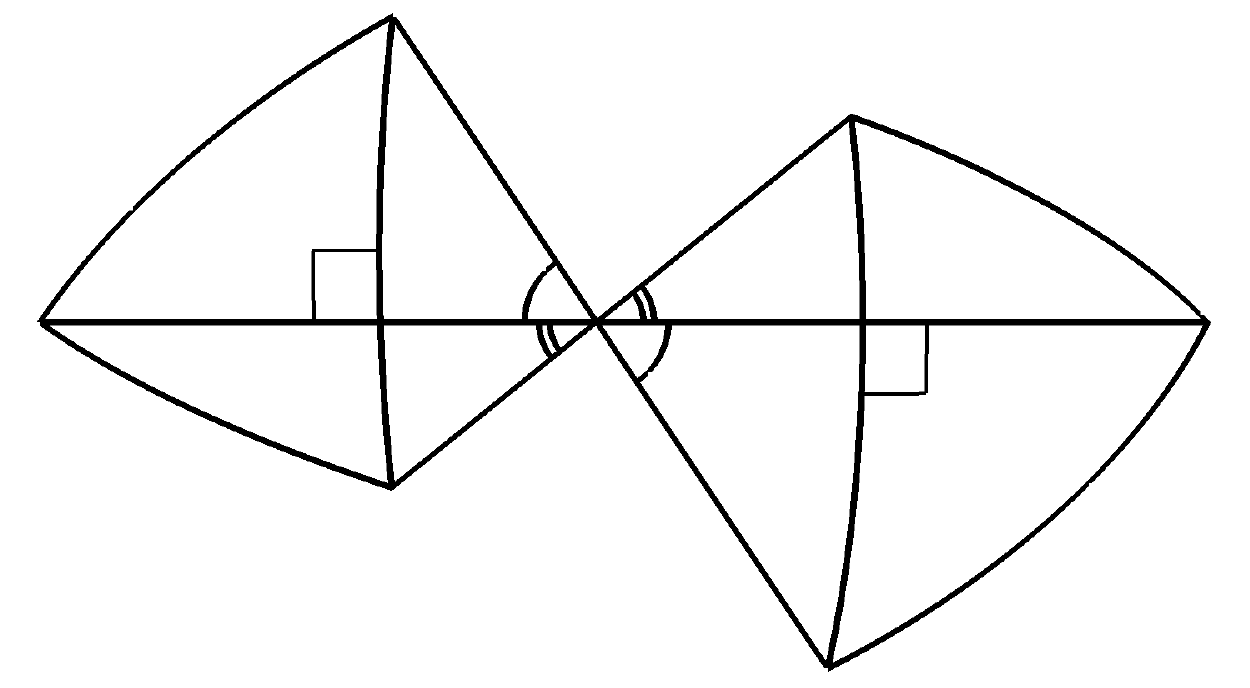}}
        \put(95,80){$\xi_2$}
        \put(137,79){$\xi_1$}
        \put(92,63){$\xi'_2$}
        \put(135,61){$\xi'_1$}
        
        \put(100,110){$\delta_2$}
        \put(137,99){$\delta_1$}
        \put(92,43){$\alpha_2$}
        \put(135,32){$\alpha_1$}
        \end{picture}
    \end{center}
    \caption{
        \label{fig:XiAngles}
        To the proof of \Href{Lemma}{lem:linkage_prop}
    }
\end{figure}

\begin{proof}
    Equation \eqref{eqn:OrthoPolynomial} implies that the involutions
    of an orthodiagonal quadrilateral are given by
    \[
        i(z,w) = (\lambda z^{-1}, w), \quad j(z,w) = (z, \mu w^{-1}).
    \]
    Therefore, the condition $\lambda_1 = \lambda_2$ is equivalent to compatibility
    of the involutions of the coupled quadrilaterals at their common vertex:
    \[
        i_1(z, w_1) = (z', w_1), \quad  i_2(z, w_2) = (z'', w_2) \Rightarrow z' = z''.
    \]
    On the other hand, the involution $i_1$ changes the angle of the first quadrilateral
    at the common vertex from $\xi_1 + \xi'_1$ to $\xi_1 - \xi'_1$,
    and the angle of the second quadrilateral from $\xi_2 + \xi'_2$ to $\xi_2 - \xi'_2$.
    The compatibility means that the angles at the common vertex
    remain equal after the involution: $\xi_1 - \xi'_1 = \xi_2 - \xi'_2$.
    Because  $\xi_1 + \xi'_1 = \xi_2 + \xi'_2$,
    this is equivalent to $\xi_1 = \xi_2$, $\xi'_1 = \xi'_2$.
    
    Observe that exchanging $\alpha$ and $\delta$ changes the sign of the involution factor.
    Thus, if we rotate the left half of the scissors linkage,
    then the condition $\lambda_1 = - \lambda_2$ transforms to $\lambda_1 = \lambda_2$.
    On the other hand, this exchanges the angles $\xi_2$ and $\xi'_2$.
    This proves the second part of the lemma.
\end{proof}

\begin{lem}
    \label{lem:SimultVanish}
    In an antiinvolutive ($\lambda_1 = -\lambda_2$) linkage of two orthodiagonal quadrilaterals,
    if one of the angles $\chi_1$ and $\chi'_2$
    in \Href{Figure}{fig:SimultVanish} becomes $0$ or $\pi$,
    then the other one does as well.
\end{lem}
\begin{proof}
Indeed, $\chi_1 = 0  \,(\mod \pi)$ if and only if $\xi_1 = 0\,(\mod \pi).$
\end{proof}

\begin{figure}[ht]
\begin{center}
\begin{picture}(150,140)
\put(-55,0){\includegraphics[scale=0.7]{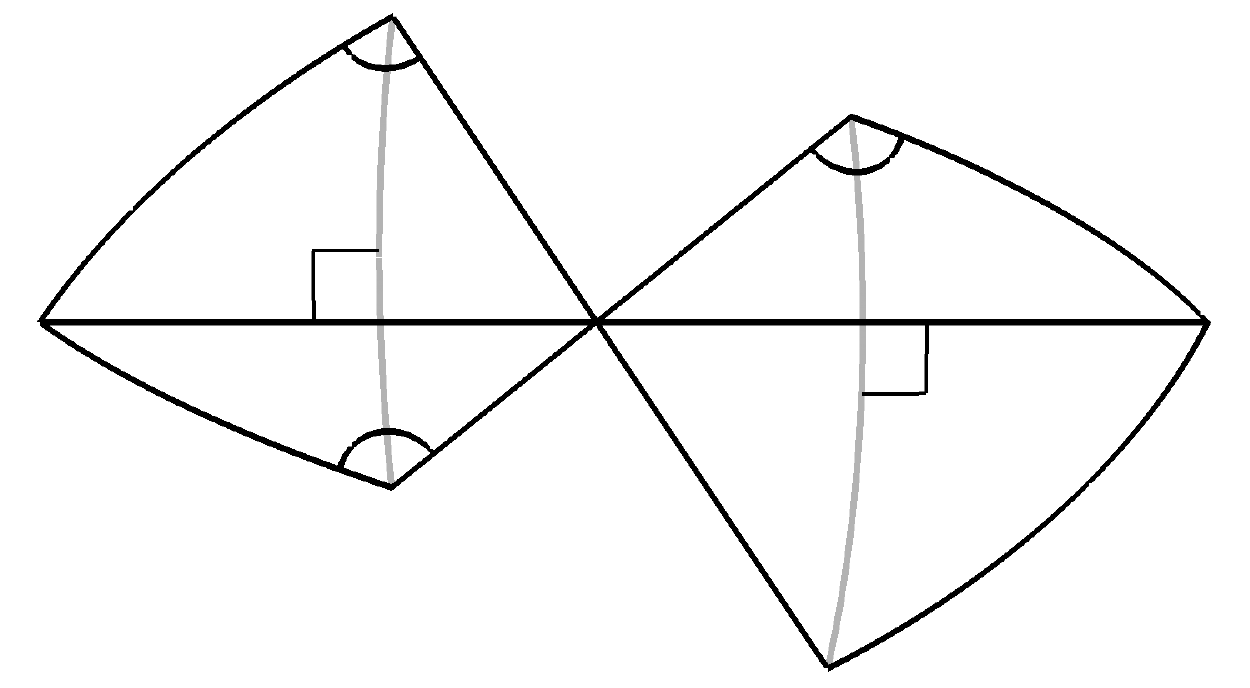}}
\put(10,119){$\chi_2$}
\put(121,99){$\chi_1$}
\put(7,56){$\chi'_2$}
\end{picture}
\end{center}
\caption{To the proof of  \Href{Lemma}{lem:SimultVanish}.}
\label{fig:SimultVanish}
\end{figure}

\subsection{Geometric properties of orthodiagonal anti-involutive type polyhedra}

From the definition of an OAI polyhedron, we see that 
all four linkages of it are  anti-involutive. Applying \Href{Lemma}{lem:SimultVanish}, we observe the following {\it flattening effect}.
\begin{cor}\label{lem:flatenning}
In a flexible  OAI Kokotsakis polyhedron  the dihedral angles at the three bold edges in \Href{Figure}{fig:PartialFlat}
vanish or become $\pi$ at the same time.
\end{cor}
\begin{figure}[ht]
\begin{center}
\includegraphics{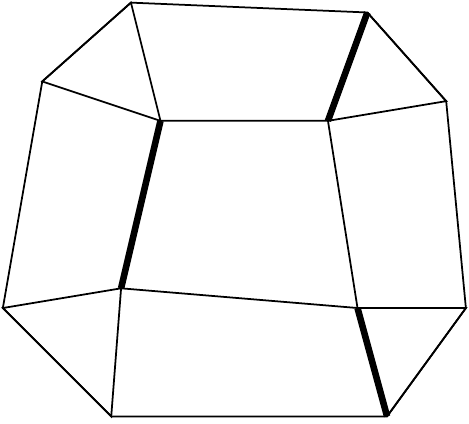}
\end{center}
\caption{A partial flattening of a flexible OAI  polyhedron.}
\label{fig:PartialFlat}
\end{figure}

Using \Href{Theorem}{thm:real_conf_space}, it is quite easy to understand what kind of flattening has a flexible OAI Kokotsakis polyhedron.
\begin{lem}\label{lem:flat_orthodiag}
Two dihedral angles of a flexible OAI Kokotsakis polyhedron   never become  $0$ or $\pi$, the two other dihedral angles become $0$ and $\pi$ on each brunch of the solution given by \Href{Theorem}{thm:real_conf_space} 
\end{lem}
\begin{proof}
We use the notation of \Href{Section}{sec:conf_space}.
If one of the angles $\phi$, $\psi_1$, $\theta$,  $\psi_2$ 
 is $0$ or $\pi$ then the corresponding  variable   is $0$ or $\infty.$ 
Since substitution \eqref{eqn:sub_for_reduced_system} just scales the variables, the same happens to the corresponding variable of system \eqref{eqn:reduced_system}. Let us consider it for simplicity. 
By \Href{Subsection}{sec:eq_positive_signs}, the variables of the first equation of \eqref{eqn:reduced_system} (with positive signature) never become $0$ or $\infty.$
Also, each branch  of the solution of the whole system contains all solutions of the first equation up to the symmetry $(\xin_1, \etan_1) \to  - (\xin_1, \etan_1).$ 
Moreover, since $\zeta_1 > 1,$ then $\xin_1 = 1$  at two different points of the solution of the first equation. 
Therefore, by the second equation,  $\etan_3$ becomes $0$ or $\infty$ at these points  on each branch of the solution of the whole system. It is easy to see that $\etan_3$ is different at this points. So it must be $0$ and $\infty$ at these points, and it can have these values only at these points. A similar argument works for $\etan_1$ and $\xin_3.$
\end{proof}
Summarizing all results, we prove Stachel's conjecture.
\begin{thm}\label{thm:Stachel_conj}
Stachel's conjecture is true. The resultant $R_{12}$ of a real flexible OAI Kokotsakis polyhedron   is reducible.  
\end{thm}
\begin{proof}
We use the notation of the first three sections. 
By \Href{Lemma}{lem:resultant}, it is enough to show that the branch sets for $w_1$ and 
$w_2$ coincide. By \Href{Lemma}{lem:InvolutionBranchedPoint}, it suffices to prove that the dihedral angles at  edges  $A_1 B_1$ and $A_2 B_2$  (see \Href{Figure}{fig:NotPlanAngles}) become 0 or $\pi$ simultaneously at four different points. This follows from \Href{Corollary}{lem:flatenning} and \Href{Lemma}{lem:flat_orthodiag}.
\end{proof}
\begin{rem}
It is easy to show that the resultant $R_{12}$ is reducible by a direct calculation of the branch sets for $w_1$ and $w_2$ from the first two equations of system \eqref{eqn:main_system_with_la_mu_signed}. The first equation of \eqref{eqn:LambdaEq} and the third equation of \eqref{eqn:SystemProporResultants} implies that the branch sets coincide. This observation works even for non-real polyhedra.
\end{rem}
\begin{lem}
\label{lem:resultant_complex}
The resultant $R^r_{12} (\etan_1, \etan_3)$ of the first two equations of system \eqref{eqn:reduced_system} is reducible if and only if $\zeta_1^2 - \zeta_2^2 = 1.$
\end{lem}
\begin{proof}
It is easy to see that $\xin_1$ belongs to the branch set of $\etan_1$ if and only if 
\[
\frac{4 \zeta_1 \xin_1}{\xin_1^2 +1} = \pm 2.
\]
From here, the branch set is
\[
\left\{\pm \zeta_1  \pm \sqrt{\zeta_1^2 -1}\right\}.
\]
Similar, the branch set of $w_2$ for the second equation is
\[
\left\{\pm \zeta_2  \pm \sqrt{\zeta_2^2 +1}\right\}.
\]
These sets coincide exactly when $\zeta_1^2 - \zeta_2^2 =1.
$
\end{proof}

\section{Example}
In this section we consider the construction of a flexible OAI Kokotsakis  polyhedron  
with angles of the base quadrilateral:
\[
\delta _1 = 1.36292,\quad
\delta _2 = 1.41009,\quad
\delta _3 = 1.80327,\quad
\delta_4 = 2\pi-\delta_1-\delta_2-\delta_3 = 1.70691.
\]
We selected $\tau = -\arctan 60 =  -1.55413 $.
Following the procedure in \Href{Theorem}{thm:algorithm}, we get 
\begin{align*}
\alpha _1&= 1.34086,& \alpha _2 &= 1.42575,&  \alpha _3&= 1.69859,& \alpha _4&= 1.81798,\\
\gamma _1&= 1.15746,& \gamma _2 &= 2.00166,& \gamma _3&= 1.4875,& \gamma _4&= 1.63656;\\
\beta _1 &= 1.11122,& \beta _2  &= 1.18397,&  \beta _3 &= 1.61684,&\beta _4&= 1.68958.
\end{align*}


In \Href{Figure}{fig:partial flattening} we show the state of the given polyhedron when we observe the flattening effect.
\begin{figure}[H]
    \includegraphics[scale=0.5]{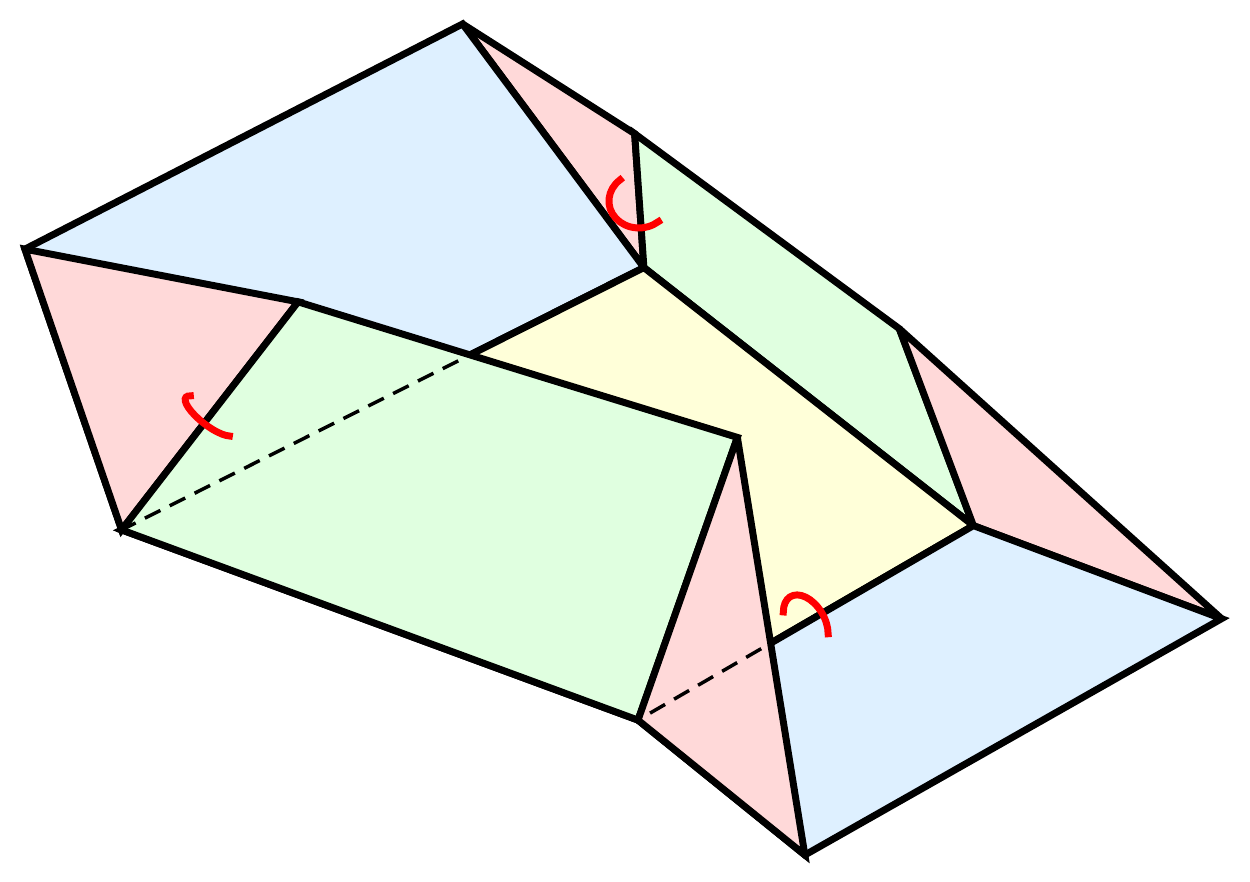}
    \caption{
        \label{fig:partial flattening}
        Partial flattening.
    }
\end{figure}
In this case, we have $k = 0.0083.$ And we observe the effect of a closeness of parameterization mentioned in \Href{Section}{sec:conf_space}. Hardly one can observe the difference between two parameterizations.
\begin{figure}[ht]
\begin{picture}(300,70)
\put(-30,0){\includegraphics[scale=0.25]{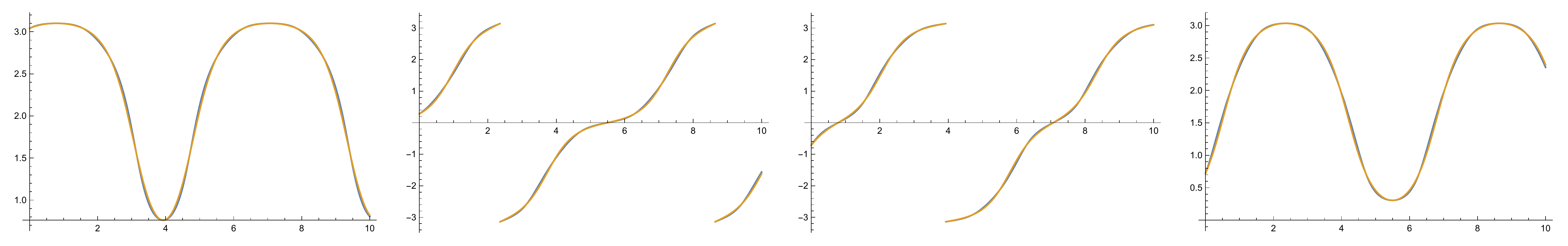}}
\end{picture}
\caption{Angles $(\phi, \psi_2, \theta, \psi_1)$ of $t$ for  a branch  
$s_{\etan} = -1, s_{\xin} = 1$ given by the parameterization of \Href{Theorem}{thm:real_conf_space} and by the parameterization of \Href{Theorem}{thm:real_conf_space_ellipttic} with linear substitution 
$t \to - \frac{K'}{\pi} t - 7.717.$ } 
\label{fig:two_parametrization}
\end{figure}
However, they are different. And we can calculate numerically the difference
\begin{figure}[ht]
\begin{picture}(300,70)
\put(-30,0){\includegraphics[scale=0.25]{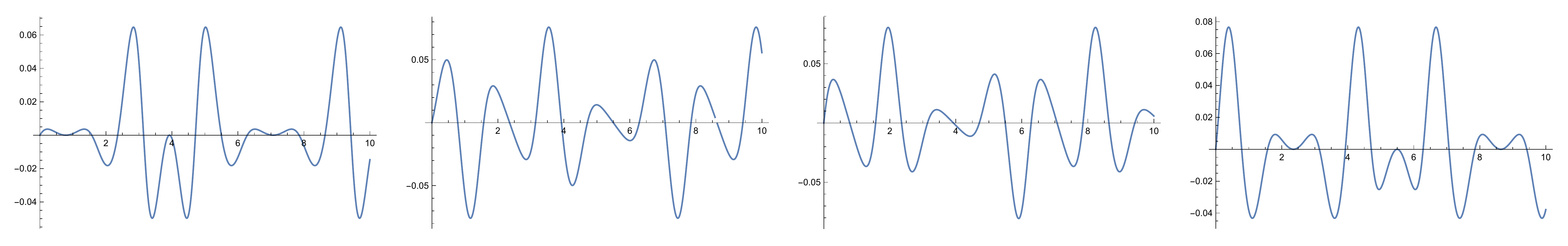}}
\end{picture}
\caption{The differences between the dihedral angles in mentioned above parameterizations.} 
\label{fig:two_parametrization_diff}
\end{figure}

\newpage
\section{Appendix: properties of orthodiagonal quadrilaterals}
\label{sec:app_orth}
\begin{lem}
\label{lem:OrthoRelation}
The diagonals of a spherical quadrilateral with side lenths $\alpha$, $\beta$, $\gamma$, $\delta$ (in this cyclic order)
are orthogonal if and only if its side lengths satisfy the relation
$$
\cos\alpha\cos\gamma = \cos\beta\cos\delta.
$$
\end{lem}
\begin{proof}
Take the intersection point of two diagonals of the quadrilateral.
(A pair of opposite vertices determines a big circle. We take one of the two intersection points of these two big circles.)
Denote the lengths of the segments between the vertices and the intersection point of the diagonals as shown in \Href{Figure}{fig:OrthoRelation}.
By the spherical Pythagorean theorem we have
$
\cos a \cos b = \cos\alpha,$ $\cos b \cos c = \cos\beta,$ $\cos c \cos d = \cos\gamma,$ $ \cos d \cos a = \cos\delta.
$
It follows that
\[
\cos\alpha\cos\gamma = \cos a \cos b \cos c \cos d = \cos\beta\cos\delta.
\]
\begin{figure}[ht]
\begin{center}
\begin{picture}(100,100)
\put(0,0){\includegraphics[scale=0.6]{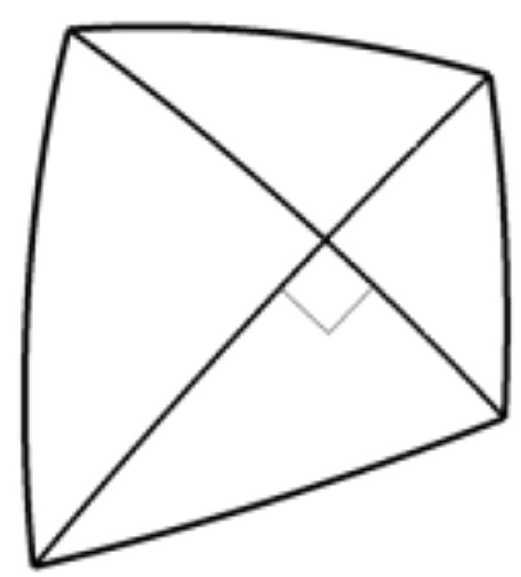}}
\put(-3,50){$\delta$}
\put(90,50){$\beta$}
\put(46,98){$\gamma$}
\put(53,8){$\alpha$}
\put(28, 36){$a$}
\put(68, 35){$b$}
\put(68, 64){$c$}
\put(41, 74){$d$}
\end{picture}
\end{center}
\caption{To the proof of \Href{Lemma}{lem:OrthoRelation}.}
\label{fig:OrthoRelation}
\end{figure}
\end{proof}

\begin{lem} \label{lem:existence_spherical_quad}
Let $\alpha, \beta, \gamma, \delta \in (0, \pi)$ satisfy $\cos \alpha \cos \gamma = \cos \beta \cos \delta.$ Then there exists a spherical orthodiagonal quadrilateral with side  lengths $\alpha, \beta, \gamma, \delta.$ 
\end{lem}
\begin{proof} Assume that $\alpha + \beta$ is the minimum of the sum of two consecutive angles in $(\alpha, \beta, \gamma, \delta, \alpha)$. If $\delta= \alpha$ and $\beta = \gamma,$ the construction is obvious. Otherwise, assume that $\delta > \alpha.$ 
Put sides $V_1 V_2$ and $V_2 V_3$ with lengths $\alpha$ and $\beta$ respectively on the same great circle.    Then there is a vertex $V_4$ such that 
$V_1 V_2 V_4$ is a right angle and   the length of $V_4 V_1$ is $\delta.$
By the identity, the length of $V_3 V_4$ is $\gamma.$
\end{proof}

\newpage
\bibliographystyle{abbrv}
\bibliography{uvolit} 

\end{document}